\documentclass{amsproc}
\usepackage{url}
\usepackage{graphicx}
\usepackage[usenames,dvipsnames]{color}

\def\cqfd{\skip10=\parfillskip\parfillskip=0pt
\enspace\hfill\symbolecqfd\par\parfillskip=\skip10\par\medskip}
\def\symbolecqfd{\rlap{$\sqcap$}$\sqcup$}

\newtheorem{theorem}{Theorem}[section]
\newtheorem{proposition}[theorem]{Proposition}
\newtheorem{lemma}[theorem]{Lemma}
\newtheorem{corollary}[theorem]{Corollary}

\theoremstyle{remark}
\newtheorem{example}[theorem]{Example}
\newtheorem{remark}[theorem]{Remark}

\def\eop{\end{proof}}

\def\binom#1#2{{#1\choose#2}}

\def\inter[#1]{[\![#1]\!]}
\def\inv{^{-1}}

\def \N {\mathbb{N}}
\def \Z {\mathbb{Z}}

\def \P {\mathbb{P}}
\def \R {\mathbb{R}}
\def \O {\mathcal{O}}
\def \C {\mathcal{C}}

\def \A {\mathcal{A}}
\def \MM {\mathbb{M}}
\def \calA {\mathcal{A}}
\def \calC {\mathcal{C}}

\let\phi\varphi
\let\epsilon\varepsilon

\def\pref{\textsf{pref}}
\def\suff{\textsf{suff}}

\def\coincidence{\alpha_{[2]}}

\newenvironment{red}{\color{red}}{}

\newcommand{\proba}[1]{\mathbb{P}\left(#1\right)}

\DeclareMathOperator{\Min}{\textsf{min}}
\DeclareMathOperator{\Lcp}{\textsf{lcp}}

\def\calG{\mathcal{G}}
\def\Red{\mathcal{R}}
\def\PP{\mathbb{P}}
\def\RR{\mathbb{R}}
\def\TT{\mathbb{T}}
\def\ZZ{\mathbb{Z}}
\def\Pref{\mathcal{P}}
\def\tuple{\mathcal{T}}
\def\tupleW{\mathcal{TW}}
\DeclareMathOperator{\Max}{\textsf{max}}
\DeclareMathOperator{\Nbr}{\textsf{size}}

\def\gothicp{\mathfrak{p}}

\begin{document}

\title[Generic properties of subgroups and presentations]{Generic properties of subgroups of free groups and finite presentations}
\thanks{The authors acknowledge partial support from ANR projects \textsc{ANR 2010 Blan 0202\_01 Frec}, \textsc{ANR 2012 JCJC JS02-012-0 MealyM} and \textsc{ANR 2010 Blan 0204\_07 Magnum}, as well as from ERC grant \textsc{PCG-336983} and the Programme IdEx Bordeaux - CPU (ANR-10-IDEX-03-02).}

%

 \author[F. Bassino]{Fr\'ed\'erique Bassino}
 \address{Universit\'e Paris 13, Sorbonne Paris Cit\'e, LIPN, CNRS UMR 7030, F-93430 Villetaneuse, France}
 \email{bassino@lipn.univ-paris13.fr}
\author[C. Nicaud]{Cyril Nicaud}
\address{Universit\'e Paris-Est, LIGM (UMR 8049), UPEMLV, F-77454 Marne-la-Vall\'ee, France}
\email{nicaud@univ-mlv.fr}
\author[P. Weil]{Pascal Weil}
\address{Univ. Bordeaux, LaBRI, CNRS UMR 5800, F-33400 Talence, France}
\email{pascal.weil@labri.fr}

\date{\today}


\begin{abstract}
Asymptotic properties of finitely generated subgroups of free groups, and of finite group presentations, can be considered in several fashions, depending on the way these objects are represented and on the distribution assumed on these representations: here we assume that they are represented by tuples of reduced words (generators of a subgroup) or of cyclically reduced words (relators). Classical models consider fixed size tuples of words (e.g. the few-generator model) or exponential size tuples (e.g. Gromov's density model), and they usually consider that equal length words are equally likely. We generalize both the few-generator and the density models with probabilistic schemes that also allow variability in the size of tuples and non-uniform distributions on words of a given length.

Our first results rely on a relatively mild prefix-heaviness hypothesis on the distributions, which states essentially that the probability of a word decreases exponentially fast as its length grows. Under this hypothesis, we generalize several classical results: exponentially generically a randomly chosen tuple is a basis of the subgroup it generates, this subgroup is malnormal and the tuple satisfies a small cancellation property, even for exponential size tuples. In the special case of the uniform distribution on words of a given length, we give a phase transition theorem for the central tree property, a combinatorial property closely linked to the fact that a tuple freely generates a subgroup.  We then further refine our results when the distribution is specified by a Markovian scheme, and in particular we give a phase transition theorem which generalizes the classical results on the densities up to which a tuple of cyclically reduced words chosen uniformly at random exponentially generically satisfies a small cancellation property, and beyond which it presents a trivial group.
\end{abstract}
    
\maketitle

This paper is part of the growing body of literature on asymptotic properties of subgroups of free groups and of finite group presentations, which goes back at least to the work of Gromov \cite{1987:Gromov} and Arzhantseva and Ol'shanskii \cite{1996:ArzhantsevaOlshanskii}. As in much of the recent literature, the accent is on so-called generic properties, that is, properties whose probability tends to 1 when the size of instances grows to infinity. A theory of genericity and its applications to complexity theory was initiated by Kapovich, Myasnikov, Schupp and Shpilrain \cite{2003:KapovichMiasnikovSchupp}, and developed in a number of papers, see Kapovich for a recent discussion \cite{2015:Kapovich}.

Genericity, and more generally asymptotic properties, depends on the fashion in which input is represented: finitely presented groups are usually given by finite presentations, i.e. tuples of cyclically reduced words; finitely generated subgroups of free groups can be represented by tuples of words (generators) or Stallings graphs. The representation by Stallings graphs is investigated by the authors, along with Martino and Ventura in \cite{2008:BassinoNicaudWeil,2013:BassinoMartinoNicaud,2015:BassinoNicaudWeil} but we will not discuss it in this paper: we are dealing, like most of the literature, with tuples of words.

There are, classically, two main models (see Section~\ref{sec: two classical models}): the few words model, where an integer $k$ is fixed and one considers $k$-tuples of words of length at most $n$, when $n$ tends to infinity, see e.g. \cite{1996:ArzhantsevaOlshanskii,2002:Jitsukawa,2013:BassinoMartinoNicaud,2015:BassinoNicaudWeil}; and the density model, where we consider tuples of cyclically reduced words of length $n$, whose size grows exponentially with $n$, see e.g. \cite{1987:Gromov,1992:Olshanskii,1995:Champetier,2004:Ollivier}.

Typical properties investigated include the following (see in particular Sections~\ref{sec: subgroups and presentations} and~\ref{sec: graphical representation}): whether a random tuple $\vec h$ freely generates the subgroup $H = \langle \vec h\rangle$ \cite{1996:ArzhantsevaOlshanskii,2002:Jitsukawa}, whether $H$ is malnormal \cite{2002:Jitsukawa,2013:BassinoMartinoNicaud} or Whitehead minimal \cite{2007:RoigVenturaWeil,2015:BassinoNicaudWeil}, whether the finite presentation with relators $\vec h$ has a small cancellation property, or whether the group it presents is infinite or trivial \cite{2004:Ollivier}.

All these models implicitly assume the uniform distribution on the set of reduced words of equal length (Ollivier also considers non-uniform distributions in \cite{2004:Ollivier}).

We introduce (Section~\ref{sec: proba model}) a model for probability distributions on tuples of reduced words that is sufficiently general to extend the few words model and Gromov's density model mentioned above, and to leave space for non uniform distributions. Like these two models, ours assumes that a tuple $\vec h$ of words is generated by independently drawing words of given lengths, but it also handles independently the size of $\vec h$ and the lengths of the words in $\vec h$.

Our first set of results assumes a \emph{prefix-heaviness} hypothesis on the probability distribution on words: the probability of drawing a word decreases exponentially fast as its length grows (precise definitions are given in Section~\ref{sec: proba model}). It is a natural hypothesis if we imagine that our probabilistic source generates words one letter at a time, from left to right. This relatively mild hypothesis suffices to obtain general results on the exponential genericity of a certain geometric property of the Stallings graph of the subgroup $H$ generated by a randomly chosen tuple $\vec h$ (the \emph{central tree property}, implicitly considered in \cite{1996:ArzhantsevaOlshanskii,2002:Jitsukawa} and explicitly in \cite{2015:BassinoNicaudWeil}), of the fact that $\vec h$ freely generates $H$, and of the malnormality of $H$, see Section~\ref{sec: general statements}. 

In Section~\ref{sec: applications to uniform}, we apply these general results to the uniform distribution and generalize known results in two directions. Firstly we consider random exponential size tuples, for which we give a phase transition theorem for the central tree property: it holds exponentially generically up to density $\frac14$, and fails exponentially generically at densities greater than $\frac14$ (Proposition~\ref{prop: uniform ctp at density}). In particular, a random tuple is exponentially generically a basis of the subgroup it generates up to density $\frac14$, but we cannot say anything of that property at higher densities.

We also extend Jitsukawa's result on malnormality \cite{2002:Jitsukawa}, from fixed size to exponential size tuples under uniform distribution up to density $\frac1{16}$ (Proposition~\ref{prop: uniform malnormal at density}). In view of the methods used to establish this result, it is likely that the value $\frac1{16}$ is not optimal.

Secondly, we show that the height of the central tree of a random fixed size tupe $\vec h$, which measures the amount of initial cancellation between the elements of $\vec h$ and $\vec h\inv$, is generically less than any prescribed unbounded non-decreasing function (Proposition~\ref{prop: very small central trees}). Earlier results only showed that this height was exponentially generically bounded by any linear function.

We then introduce \emph{Markovian automata}, a probabilistic automata-theoretic model, to define explicit instances of prefix-heavy distributions (Section~\ref{sec: automata probabilities}). Additional assumptions like irreducibility or ergodicity lead to the computation of precise bounds for the parameters of prefix-heaviness. In particular, we prove a phase transition theorem for ergodic Markovian automata (Section~\ref{sec: phase transitions}), showing that small cancellation properties generically hold up to a certain density, and generically do not hold at higher densities. More precisely, if $\coincidence$ is the coincidence probability of the Markovian automaton, Property $C'(\lambda)$ holds exponentially generically at $\coincidence$-density less than $\frac\lambda2$ (that is: for random tuples of size $\coincidence^{-dn}$ for some $d < \frac\lambda2$), and fails exponentially generically at $\coincidence$-densities greater than $\frac\lambda2$. We also show that at $\coincidence$-densities greater than $\frac12$, a random tuple of cyclically reduced words generically presents a degenerate group (see Proposition~\ref{prop: degenerate subgroups} for a precise definition). These results generalize the classical results on uniform distribution in Ollivier \cite{2004:Ollivier,2005:Ollivier}. It remains to be seen whether our methods can be applied to fill the gap, say, between $\coincidence$-density $\frac1{12}$ and $\frac12$, where small cancellation property $C'(\frac16)$ generically does not hold yet the presented group might be hyperbolic, see \cite{2004:Ollivier,2005:Ollivier}.

Some of the definitions in this paper, notably that of Markovian automata, were introduced by the authors in \cite{2012:BassinoNicaudWeil}, and some of the results were announced there as well. The results in the present paper are more precise, and subsume those of \cite{2012:BassinoNicaudWeil}.

\section{Free groups, subgroups and presentations}

In this section, we set the notation and basic definitions of the properties of subgroups of free groups and finite presentations which we will consider.

\subsection{Free groups and reduced words}\label{sec: free groups}

Let $A$ be a finite non-empty set, which will remain fixed throughout the paper, with $|A| = r$, and let $\tilde A$ be the symmetrized alphabet, namely the disjoint union of $A$ and a set of formal inverses $A\inv = \{a\inv\in A \mid a\in A\}$.  By convention, the formal inverse operation is extended to $\tilde A$ by letting $(a\inv)\inv=a$ for
each $a \in A$.  A word in $\tilde A^*$ (that is: a word written on the alphabet $\tilde A$) is
\emph{reduced} if it does not contain length 2 factors of the form $aa\inv$ ($a\in \tilde A$).  If a word is not reduced, one can \emph{reduce} it by iteratively deleting every
factor of the form $aa\inv$.  The resulting reduced word is uniquely
determined: it does not depend on the order of the cancellations.  For
instance, $u=aabb\inv a\inv$ reduces to $aa a\inv$, and thence to $a$.

The set $F$ of reduced words is naturally equipped with a group structure, where the product $u\cdot v$ is the (reduced) word obtained by
reducing the concatenation $uv$.  This group is called the
\emph{free group} on $A$.  More generally, every group isomorphic to
$F$, say, $G = \phi(F)$ where $\phi$ is an isomorphism, is said to be
a free group, \emph{freely generated by} $\phi(A)$.  The set $\phi(A)$ is
called a \emph{basis} of $G$.  Note that if $r\ge 2$, then $F$ has
infinitely many bases: if, for instance, $a\ne b$ are elements of $A$, then replacing $a$ by $b^nab^m$ (for some integers $n,m$) yields a basis.  The \emph{rank} of $F$ (or
of any isomorphic free group) is the cardinality $|A|$ of $A$, and one
shows that this notion is well-defined in the following sense: every basis of $F$ has the same cardinality.

Let $x, y$ be elements of a group $G$. We say that $y$ is a \emph{conjugate} of 
$x$ if there exists an element $g\in G$ such that $y = g\inv xg$, which we write $y = x^g$. The notation is extended to subsets of $G$: if $H \subseteq G$, then $H^g = \{x^g \mid x\in H\}$. Conjugacy of elements of the free group $F$ is characterized as follows. Say that a word $u$ is  \emph{cyclically reduced word} if it is non-empty, reduced and its first and last letters are not mutually inverse (or equivalently, if $u^2$ is non-empty and reduced). For instance, $ab\inv a\inv bbb$ is cyclically reduced, but $ab\inv a\inv bba\inv$ is not.

For every reduced word $u$, let $\kappa(u)$ denote its \emph{cyclic
reduction}, which is the shortest word $v$ such that $u=wvw^{-1}$ for some word $w$. For instance, $\kappa(ab\inv a\inv bba\inv) = a\inv b$. It is easily verified that two reduced words $u$ and $v$ are conjugates if and only if $\kappa(u)$ and $\kappa(v)$ are \emph{cyclic conjugates} (that is: there exist words $x$ and $y$ such that $\kappa(u) = xy$ and $\kappa(v) = yx$).

Let $\Red_{n}$ (resp. $\calC_n$) denote the set of all reduced (resp. cyclically reduced) words of length $n\geq 1$, and let $\Red = \bigcup_{n\geq 1}\Red_{n}$ and $\calC = \bigcup_{n\geq 1}\calC_{n}$ be the set of all reduced words, and all cyclically reduced words, respectively.

Every word of length 1 is cyclically reduced, so $|\Red_1| = |\calC_1| = 2r$. A reduced word of length $n\ge 2$ is of the form $ua$, where $u$ is reduced and  $a$ is not the inverse of the last letter of $u$. An easy induction shows that there are $|\Red_{n}|=2r(2r-1)^{n-1} = \frac{2r}{2r-1}(2r-1)^n$ reduced words of length $n\geq 2$.

Similarly, if $n\ge 2$, then $\calC_n$ is the set of words of the form $ua$, where $u$ is a reduced word and $a\in\tilde A$ is neither the inverse of the first letter of $u$, nor the inverse of its last letter: for a given $u$, there are either $2r-1$ or $2r-2$ such words, depending whether the first and last letter of $u$ are equal. In particular, the number of words in $\calC_n$ satisfies $\frac{2r}{2r-1}(2r-1)^{n-1}(2r-2) \le |\calC_n| \le \frac{2r}{2r-1}(2r-1)^n$, and in particular, $|\calC_n| = \Theta((2r-1)^n)$.

\subsection{Subgroups and presentations}\label{sec: subgroups and presentations}

Given a tuple $\vec h = (h_1,\ldots,h_k)$ of elements of $F$, let $\vec h^\pm = (h_1,h_1\inv,\ldots,h_k,h_k\inv)$ and let $\langle\vec h\rangle$ denote the subgroup of $F$ generated by the elements of $\vec h$, that is, the set of all the elements of $F$ which can be written as a product of elements of $\vec h^\pm$. It is a classical result of Nielsen that every such subgroup is free \cite{1918:Nielsen}.

An important property of subgroups is malnormality, which is related to geometric considerations (\textit{e.g.} \cite{1998:GitikMitraRips,1998:KharlampovichMyasnikov}): a subgroup $H$ of a group $G$ is \emph{malnormal} if $H \cap H^x$ is trivial for every $x\not\in H$. It is decidable whether a finitely generated subgroup $\langle\vec h\rangle$ is malnormal (\cite{2002:Jitsukawa,2002:KapovichMyasnikov}, see Section~\ref{sec: graphical representation}), whereas malnormality is not decidable in general hyperbolic groups \cite{2001:BridsonWise}.

A tuple $\vec h$ of elements of $F(A)$ can also be considered as a set of relators in a group presentation. More precisely, we denote by $\langle A \mid \vec h\rangle$ the group with generator set $A$ and relators the elements of $\vec h$, namely the quotient of $F(A)$ by the normal subgroup generated by $\vec h$. It is customary to consider such a group presentation only when $\vec h$ consists only of cyclically reduced words, since $\langle A \mid \vec h\rangle = \langle A \mid \kappa(\vec h)\rangle$.

The small cancellation property is a combinatorial property of a group presentation, with far-reaching consequences on the quotient group. Let $\vec h$ be a tuple of cyclically reduced words. A \emph{piece} in $\vec h$ is a word $u$ with at least two occurrences as a prefix of a cyclic conjugate of a word in $\vec h^\pm$. Let $0 < \lambda < 1$. The tuple $\vec h$ (or the group presentation $\langle A \mid \vec h\rangle$) has the \emph{small cancellation property $C'(\lambda)$} if whenever a piece $u$ occurs as a prefix of a cyclic conjugate $w$ of a word in $\vec h^\pm$, then $|u| < \lambda|w|$.

The following properties are well-known. We do not give the definition of the group-theoretic properties in this statement and refer the reader to \cite{1977:LyndonSchupp} or to the comprehensive survey \cite{2005:Ollivier}. 

\begin{proposition}\label{prop: small cancellation properties}
If $\vec h$ is a tuple of cyclically reduced words satisfying $C'(\frac16)$, then $G = \langle A\mid \vec h\rangle$ is infinite, torsion-free and word-hyperbolic. In addition, it has solvable word problem (by Dehn's algorithm) and solvable conjugacy problem. 

Moreover, if $\vec h$ and $\vec g$ both have property $C'(\frac16)$ and if they present the same group, then $\vec h^\pm = \vec g^\pm$ up to the order of the elements in the tuples.
\end{proposition}

\subsection{Graphical representation of subgroups and the central tree property}\label{sec: graphical representation}

A privileged tool for the study of subgroups of free groups is provided by \emph{Stallings graphs}: if $H$ is a finitely generated subgroup of $F$, its Stallings graph $\Gamma(H)$ is a finite graph of a particular type, uniquely representing $H$, whose computation was first made explicit
by Stallings \cite{1983:Stallings}. The mathematical object itself is
already described by Serre \cite{1980:Serre}. The description we give
below differs slightly from Serre's and Stallings', it follows
\cite{2000:Weil,2002:KapovichMyasnikov,2006:Touikan,2007:MiasnikovVenturaWeil,2008:SilvaWeil}
and it emphasizes the combinatorial, graph-theoretical aspect, which is
more conducive to the discussion of algorithmic properties.

A \textit{finite $A$-graph} is a pair $\Gamma = (V,E)$ with $V$ finite
and $E \subseteq V\times A\times V$, such that if both $(u,a,v)$ and
$(u,a,v')$ are in $E$ then $v = v'$, and if both $(u,a,v)$ and
$(u',a,v)$ are in $E$ then $u = u'$.
Let $v\in V$. The pair $(\Gamma,v)$ is said to be \emph{admissible} if
the underlying graph of $\Gamma$ is connected (that is: the undirected
graph obtained from $\Gamma$ by forgetting the letter labels and the
orientation of edges), and if every vertex $w\in V$, except possibly
$v$, occurs in at least two edges in $E$.

Every admissible pair $(\Gamma,1)$ represents a unique subgroup $H$ of $F(A)$ in the following sense: if $u$ is a reduced word, then $u\in H$ if and only if $u$ labels a loop at 1 in
$\Gamma$ (by convention, an edge $(u,a,v)$ can be read from $u$ to $v$
with label $a$, or from $v$ to $u$ with label $a\inv$). One can show that $H$ is finitely generated. More precisely, the following procedure yields a basis of $H$: choose a spanning tree $T$ of $\Gamma$; for each edge $e = (u,a,v)$ of $\Gamma$ not in $T$, let $b_e = x_uax_v\inv$, where $x_u$ (resp. $x_v$) is the only reduced word labeling a path in $T$ from 1 to $u$ (resp. $v$); then the $b_e$ freely generate $H$ and as a result, the rank of $H$ is exactly $|E| -|V| + 1$.

Conversely, if $\vec h = (h_1,\ldots,h_k)$ is a tuple of reduced words, then the subgroup $H = \langle\vec h\rangle$ admits a Stallings graph, written $(\Gamma(H),1)$, which can be computed effectively and efficiently. A quick description of the algorithm is as follows. We first build a graph with edges labeled by letters in $\tilde A$, and then reduce it to an $A$-graph using \emph{foldings}.  First build a vertex $1$.  Then, for every $1\le i\le k$,
build a loop with label $h_i$ from $1$ to $1$, adding $|h_i|-1$ new vertices.
Change every edge $(u,a\inv,v)$ labeled by a letter of $A\inv$ into an
edge $(v,a,u)$. At this point, we have constructed the so-called \emph{bouquet of loops} labeled by the $h_i$.

Then iteratively identify the vertices $v$ and $w$
whenever there exists a vertex $u$ and a letter $a\in A$ such that
either both $(u,a,v)$ and $(u,a,w)$ or both $(v,a,u)$ and $(w,a,u)$
are edges in the graph (the corresponding two edges are
\textit{folded}, in Stallings' terminology).

The resulting graph $\Gamma$ is such that $(\Gamma,1)$ is admissible, the reduced words labeling a loop at 1 are exactly the elements of $H$
and, very much like in the (1-dimensional) reduction of words, that graph does
not depend on the order used to perform the foldings. The graph $(\Gamma(H),1)$ can  be computed in time almost linear (precisely: in time $\O(n\log^*n)$ \cite{2006:Touikan}).

Some algebraic properties of $H$ can be directly seen on its Stallings
graph $(\Gamma(H),1)$. For instance, one can show that $H$ is malnormal if and only if there exists no non-empty reduced word $u$ which labels a loop in two distinct vertices of $\Gamma(H)$ \cite{2002:Jitsukawa,2002:KapovichMyasnikov}. This property leads to an easy decision procedure of malnormality for subgroups of a free group. We refer the reader to~\cite{1983:Stallings,2000:Weil,2002:KapovichMyasnikov,2007:MiasnikovVenturaWeil} for more information about Stallings graphs.

If $\vec h$ is a tuple of elements of $F$, let $\Min(\vec h)$ be the minimum length of an
element of $\vec h$ and let $\Lcp(\vec h)$ be the length of the longest common prefix between two words in $\vec h^\pm$, see Figure~\ref{fig:central tree}\footnote{This definition is closely related with the notion of \emph{trie} of $\vec h^\pm$. The height of the trie of $\vec h^\pm$ is $1 + \Lcp(\vec h)$.}. We
say that $\vec h$ has the \emph{central tree property} if $2 \Lcp(\vec
h) < \Min(\vec h)$. 
\begin{figure}[htbp]
\centerline{\includegraphics[scale=1.1]{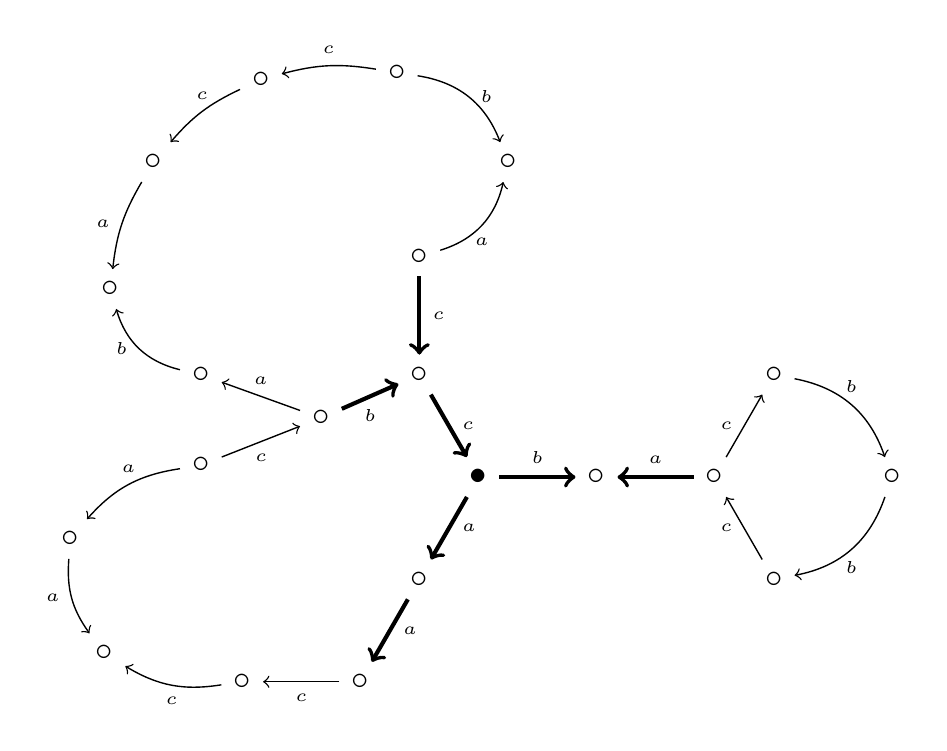}}
\caption{
The Stallings graph of the subgroup generated by $\vec h = (ba\inv cb^2a^2b\inv, a^2c^2a^{-2}cbc, c\inv b\inv aba\inv c^{-2}ba\inv c^2)$, has the central tree property and satisfies $\Lcp(\vec h) = 2$. The origin is denoted by $\bullet$ and the central tree is depicted in bold arrows.
}
\label{fig:central tree}
\end{figure}

\begin{proposition}\label{fact: central tree}
Let $\vec h = (h_1,\ldots, h_k)$ be a tuple of elements of $F(A)$ with the central tree property and let $H = \langle \vec h\rangle$. Then the Stallings graph $\Gamma(H)$ consists of a \emph{central tree} of height $t=\Lcp(\vec h)$ and of $k$ \emph{outer loops}, one for each $h_i$, connecting the length $t$ prefix and the length $t$ suffix of $h_i$ (two leaves of the central tree), of length $|h_i|-2t$ respectively. The set of vertices of the central tree can be identified with the set of prefixes of length at most $t$ of the words of $\vec h^\pm$.

In particular, $\vec h$ is a basis of $H$. Moreover, if $\vec g$ is a basis of $H$ also with the central tree property, then $\vec h^\pm$ and $\vec {g}^\pm$ coincide up to the order of their elements.
\end{proposition}

\begin{proof}
The central tree property shows that the cancellation (folding) that occurs when one considers the bouquet of $h_i$-labeled loops around the origin, stops before canceling entirely any one of the $h_i$. The result follows immediately.
\eop

Under the central tree property, we record an interesting sufficient condition for malnormality.

\begin{proposition}\label{prop: sufficient malnormality}
Let $\vec h = (h_1,\ldots, h_k)$ be a tuple of elements of $F(A)$ with the central tree property and let $H = \langle \vec h\rangle$. Let us assume additionally that $3\Lcp(\vec h) < \Min(\vec h)$ and that no word of length at least $\frac12(\Min(\vec h) - 3\Lcp(\vec h))$ has several occurrences as a factor of an element of $\vec h^\pm$, then $H$ is malnormal.
\end{proposition}

\begin{remark}
In the proof below, and in several other statements and proofs later in the paper, we consider words whose length is specified by an algebraic expression which does not always compute to an integer (\textit{e.g.}, $\frac12(\Min(\vec h) - 3\Lcp(\vec h))$). To be rigorous, we should consider only the integer part of these expressions. For the sake of simplicity, we dispense with this extra notation, and implicitly consider that if a word of length $\ell$ is considered, then we mean that its length is $\lfloor \ell\rfloor$.
\end{remark}

\begin{proof}
Let $m = \Min(\vec h)$ and $t = \Lcp(\vec h)$. Proposition~\ref{fact: central tree} shows that $\Gamma(H)$ consists of a central tree of height $t$ and of outer loops, one for each $h_i$, of length $|h_i|-2t \ge m-2t$.

If $H$ is not malnormal, then a word $u$ labels a loop at two distinct vertices of $\Gamma(H)$.  Without loss of generality, $u$ is cyclically reduced. Moreover, given the particular geometry of $\Gamma(H)$, both loops visit the central tree. Without loss of generality, we may assume that one of the $u$-labeled loops starts in the central tree, at distance exactly $t$ from the base vertex 1, and travels away from 1. In particular, $|u| \ge m-2t$, and if $v$ is the prefix of $u$ of length $m-2t$, then $v$ is a factor of some $h_i^{\pm1}$.

Let $s$ be the start state of the second $u$-labeled loop: reading this loop starts with reading the word $v$. Suppose that $s$ is in the central tree: either reading $u$ (and $v$) from $s$ takes us away from $1$ towards a leaf of the central tree and into an outer loop, and $v$ is a factor of some $h_j^{\pm1}$; or reading $v$ from $s$ moves us towards 1 for a distance at most $t$, after which the path travels away from 1, along a path labeled by a factor of some $h_j^{\pm1}$, for a distance at least $m-3t$. In either case, a factor of $u$ of length $m-3t > \frac12(m-3t)$ has two occurrences in $\vec h^\pm$.

Suppose now that $s$ is on an outer loop (say, associated to $h_j^{\pm1}$) and that $s'$ is the first vertex of the central tree reached along the loop. If $s'$ is reached after reading a prefix of $u$ of length greater than $\frac12(m-3t)$, then the prefix of $v$ of length $\frac12(m-3t)$ is a factor of $h_j^{\pm1}$. Otherwise $v$ labels a path from $s$ which first reaches $s'$, then travels towards 1 in the central tree for a distance at most $t$, and thence away from 1, along a path labeled by some $h_\ell^{\pm1}$, which it follows over a length at least equal to $(m - 2t) - \frac12(m - 3t) - t = \frac12(m - 3t)$.

Thus, in every case, $u$ contains a factor of length $\frac12(m - 3t)$ with two distinct occurrences as a factor of an element of $\vec h^\pm$ and this concludes the proof.
\eop

To conclude this section, we note that the properties discussed above are preserved when going from a tuple $\vec h$ to a sub-tuple: say that a tuple $\vec g$ is \emph{contained in} a tuple $\vec h$, written $\vec g \le \vec h$, if every element of $\vec g$ is an element of $\vec h$.

\begin{proposition}\label{prop: inherited properties}
Let $\vec g, \vec h$ be tuples of reduced words such that $\vec g \le \vec h$.
\begin{itemize}
\item If $\vec h$ has the central tree property, so does $\vec g$.

\item If $\vec h$ consists of cyclically reduced words and $\vec h$ has Property $C'(\lambda)$, then so does $\vec g$.

\item If $\vec h$ has the central tree property, then $\langle\vec g\rangle$ is a free factor of $\langle\vec h\rangle$, and $\langle\vec g\rangle$ is malnormal if $\langle\vec h\rangle$ is.
\end{itemize}
\end{proposition}

\begin{proof}
The first two properties are immediate from the definition.
Supose now that $\vec h$ has the central tree property. Then by Proposition~\ref{fact: central tree}, $\vec h$ is a basis of $\langle\vec h\rangle$, and by the first statement of the current proposition, $\vec g$ is a basis of $\langle\vec g\rangle$. Since $\vec g \le \vec h$, $\langle\vec g\rangle$ is a free factor of $\langle\vec h\rangle$.

In particular, $\langle\vec g\rangle$ is malnormal in $\langle\vec h\rangle$ (a free factor always is, by elementary reasons). It is immediate from the definition that malnormality is transitive, so if $\langle\vec h\rangle$ is malnormal in $F$, then so is $\langle\vec g\rangle$.
\eop

\section{Random models and generic properties}

We will discuss several models of randomness for finitely presented groups and finitely generated subgroups, or rather, for finite tuples of cyclically reduced words (group presentations) and finite tuples of reduced words. In this section, we fix a general framework for these models of randomness and we survey some of the known results.

\subsection{Generic properties and negligible properties}

Let us say that a function $f$, defined on $\N$ and such that $\lim f (n) = 0$, is \emph{exponentially} (resp. \emph{super-polynomially}, \emph{polynomially}) \emph{small} if $f(n) = o(e^{-dn})$ for some $d > 0$ (resp. $f(n) = o(n^{-d})$ for every positive integer $d$, $f(n) = o(n^{-d})$ for some positive integer $d$).

Given a sequence of probability laws $(\P_n)_n$ on a set $S$, we say that a subset $X \subseteq S$ is \emph{negligible} if $\lim_n \P_n(X) = 0$, and \emph{generic} if its complement is negligible.%
\footnote{This is the same notion as \emph{with high probability} or \emph{with overwhelming probability}, which are used in the discrete probability literature.}

We also say that $X$ is \emph{exponentially} (resp. \emph{super-polynomially}, \emph{polynomially}) \emph{negligible} if $\P_n(X)$ tends to $0$ and is exponentially (resp. super-polynomially, polynomially) small. And it is \emph{exponentially} (resp. \emph{super-polynomially}, \emph{polynomially}) \emph{generic} if its complement is exponentially (resp. super-polynomially, polynomially)  negligible.

In this paper, the set $S$ will be the set of all finite tuples of reduced words, or cyclically reduced words, and the probability laws $\P_n$ will be such that every subset is measurable: we will therefore not specify in the statements that we consider only measurable sets.

The notions of genericity and negligibility have elementary closure properties that we will use freely in the sequel. For instance, a superset of a generic set is generic, as well as the intersection of finitely many generic sets. Dual properties hold for negligible sets.

\subsection{The few-generator model and the density model}\label{sec: two classical models}

In this section, we review the results known on two random models, originally introduced to discuss finite presentations. We discuss more general models in Section~\ref{sec: proba model} below.

\subsubsection{The few-generator model}\label{sec: few-generator model}

In the \emph{few-generator model}, an integer $k \ge 1$ is fixed, and we let $\P_n$ be the uniform probability on the set of $k$-tuples of words of $F$ of length at most $n$. Proposition~\ref{prop: few generators basic} is established by elementary counting arguments, see Gromov \cite[Prop. 0.2.A]{1987:Gromov} or Arzhantseva and Ol'shanskii \cite[Lemma 3]{1996:ArzhantsevaOlshanskii}.

\begin{proposition}\label{prop: few generators basic}
Let $k \ge 1$, $0 < \alpha < \frac12$, $2\alpha < \beta < 1$ and $ 0 < \lambda < 1$. Then a $k$-tuple $\vec h$ of elements of $F$ of length at most $n$ picked uniformly at random, exponentially generically satisfies the following properties:
\begin{itemize}
\item $\Min(\vec h) > \beta n$,

\item $\Lcp(\vec h) < \alpha n$,

\item no word of length $\lambda n$ has two occurrences as a factor of an element of $\vec h^\pm$.
\end{itemize}
\end{proposition}
In view of Propositions~\ref{fact: central tree} and \ref{prop: sufficient malnormality}, this yields the following corollary (\cite{2013:BassinoMartinoNicaud}, and \cite{2002:Jitsukawa} for the malnormality statement).
\begin{corollary}\label{cor: few generators}
Let $k \ge 1$. If $\vec h$ is a $k$-tuple  of elements of $F$ of length at most $n$ picked uniformly at random and $H = \langle\vec h\rangle$, then 
\begin{itemize}
\item exponentially generically, $\vec h$ has the central tree property, and in particular, $\Gamma(H)$ can be constructed in linear time (in $k \cdot n$), simply by computing the initial cancellation of the elements of $\vec h^\pm$; $H$ is freely generated by the elements of $\vec h$, and $H$ has rank $k$;

\item exponentially generically, $H$ is malnormal.
\end{itemize}
Moreover, if $\vec h$ and $\vec g$ generate the same subgroup, then exponentially generically, $\vec h^\pm = \vec {g}^\pm$ up to the order of the elements in the tuples.
\end{corollary}

The following statement follows from Proposition~\ref{prop: inherited properties}, and from Theorem~\ref{thm: phase transitions} below (which is independent).

\begin{corollary}\label{cor: few generators small cancellation}
In the few-generator model, if $\vec h$ is a $k$-tuple of cyclically reduced words of length at most $n$, then
\begin{itemize}
\item for any $0 < \lambda < \frac12$, $\vec h$ exponentially generically satisfies the small cancellation property $C'(\lambda)$ ;

\item exponentially generically, the group $\langle A \mid \vec h\rangle$ is infinite, torsion-free, word-hyperbolic, it has solvable word problem (by Dehn's algorithm) and solvable conjugacy problem.
\end{itemize}
\end{corollary}

\subsubsection{The density model}\label{sec: density model}

In the \emph{density model}, a density $0 < d < 1$ is fixed, and a tuple of cyclically reduced elements of the $n$-sphere of density $d$ is picked uniformly at random: that is, the tuple $\vec h$ consists of $|\calC_n|^d$ cyclically reduced words of length $n$. This model was introduced by Gromov \cite{1993:Gromov} and complete proofs were given by Ol'shanskii \cite{1992:Olshanskii}, Champetier \cite{1995:Champetier} and Ollivier \cite{2004:Ollivier}.

\begin{theorem}\label{thm: phase transitions}
Let $0 < \alpha < d < \beta < 1$. In the density model, the following properties hold:
\begin{enumerate}
%
%
\item exponentially generically, every word of length $\alpha n$ occurs as a factor of a word in $\vec h$, and some word of length $\beta n$ fails to occur as a factor of a word in $\vec h^\pm$;

\item if $d < \frac12$, then exponentially generically, $\vec h$ satisfies property $C'(\lambda)$ for $\lambda > 2d$ but $\vec h$ does not satisfy $C'(\lambda)$ for $\lambda < 2d$; in particular, at density $d < \frac1{12}$, $\vec h$ satisfies exponentially generically property $C'(\frac16)$ and the group $\langle A \mid \vec h\rangle$ is infinite and hyperbolic; and at density $d > \frac1{12}$, exponentially generically, $\vec h$ does not satisfy $C'(\frac16)$;

\item at density $d > \frac12$, exponentially generically, $\langle\vec h\rangle$ is equal to $F(A)$, or has index 2. In particular, the group $\langle A \mid \vec h\rangle$ is either trivial or $\Z/2\Z$;

\item at density $d < \frac12$, the group $\langle A \mid \vec h\rangle$ is generically infinite and hyperbolic.
\end{enumerate}
\end{theorem}

Properties (1)-(3) in Theorem~\ref{thm: phase transitions} are obtained by counting arguments. Property (4) is the ``hard part'' of the theorem, where hyperbolicity does not follow from a small cancellation property.

As pointed out by Ollivier \cite[Sec. I.2.c]{2005:Ollivier}, the statement of Theorem~\ref{thm: phase transitions} still holds if a tuple of cyclically reduced elements is chosen uniformly at random at density $d$ in the $n$-ball rather than in the $n$-sphere (that is, it consists of words of length at most $n$). We will actually verify this fact again in Section~\ref{sec: applications to uniform}.

\section{A general probabilistic model}\label{sec: proba model}

We introduce a fairly general probabilistic model, which generalizes both the few-generator and the density models.

\subsection{Prefix-heavy sequences of measures on reduced words}
For every reduced word $u \in \Red $, let $\Pref(u)$ be the set of all reduced
words $v$ of which $u$ is a prefix (that is: $\Pref(u) = u\tilde A^* \cap \Red$). Let also $\Pref_{n}(u)$ be
the set $\Red_{n}\cap\Pref(u)$. The notation $\Pref$ can also be extended to a set $U$ of reduced words: $\Pref(U) = \bigcup_{u\in U}\Pref(u)$.

Let $(\RR_{n})_{n\geq 0}$ be a sequence of probability measures on $\Red$ and let $C\geq 1$ and $\alpha\in(0,1)$. We say that
the sequence  $(\RR_{n})_{n\geq 0}$ is a \emph{prefix-heavy sequence of measures on $\Red$} of parameters $(C,\alpha)$ if:
\begin{enumerate}
\item for every $n\geq 0$, the support of the measure $\RR_{n}$ is included
in $\Red_{n}$;

\item for every $n\geq 0$ and for every $u\in\Red$, if $\RR_{n}(\Pref(u))\neq 0$
then for every $v\in\Red$
\[
\RR_{n}\big(\Pref(uv)\mid\Pref(u)\big) \leq C \alpha^{|v|}.
\]
\end{enumerate}
This prefix-oriented definition is rather natural if one thinks of a source as generating reduced words from left to right, as is usual in information theory. 

\begin{remark}
Taking $u=\epsilon$ in the definition yields
$\RR_{n}\big(\Pref(v)\big) \leq C \alpha^{|v|}$. For
$n = |v|$, we have $\Pref(v)\cap\Red_{n}=\{v\}$, so the probability
of $v$ decreases exponentially with the length of $v$.
\end{remark}

\begin{example}\label{ex: uniform distribution on words}
The sequence of uniform distributions on $\Red_{n}$ is a prefix-heavy
sequence of measures with parameters $C=1$ and $\alpha=\frac1{2r-1}$. Indeed, if $u$ is a reduced word of length at most $n\geq 0$ (for a longer $u$, $\RR_{n}(\Pref(u))=0$), and if $uv$ is reduced, we have
\[
\RR_{n}\big(\Pref(uv)\mid\Pref(u)\big) =
\begin{cases}
\frac1{(2r-1)^{|v|}} & \text{if }|u|+|v|\leq n \text{ and }u\neq\epsilon,\\
\frac1{2r(2r-1)^{|v|-1}} & \text{if }|v|\leq n\text{ and } u=\epsilon,\\
0 & \text{otherwise}.
\end{cases}
\]
\end{example}

\begin{example}
By a similar computation, one verifies that the sequence of uniform distributions on $\C_{n}$, the cyclically reduced words, is also a prefix-heavy sequence of measures, with parameters  $C = \frac{2r-1}{2r-2}$ and $\alpha = \frac1{2r-1}$ (see Section~\ref{sec: free groups}).
\end{example}

For the rest of this section, we fix a  sequence of measures $(\RR_{n})_{n\geq 0}$ on $\Red$, which is prefix-heavy with parameters $(C,\alpha)$. All probabilities refer to this sequence, that is: the probability of a subset of $\Red_n$ is computed according to $\RR_n$.

\begin{remark}
If $X$ and $Y$ are subsets of $\Red$, the notation $\RR_n(X \mid Y)$ is technically defined only if $\RR_n(Y) \ne 0$. To avoid stating cumbersome hypotheses, we adopt the convention that $\RR_n(X \mid Y)\ \RR_n(Y) = 0$ whenever $\RR_n(Y) = 0$.
\end{remark}

\subsection{Repeated factors in random reduced words}\label{sec: repeated factors}

Let us first evaluate the probability of occurrence of prescribed, non-overlapping factors in a reduced word. Let $m \ge 0$, $\vec v = (v_1,\ldots,v_m)$ be a vector of non-empty reduced words and $\vec\imath = (i_1,\ldots,i_m)$ be a vector of integers. We denote by $E(\vec v,\vec\imath)$ denote the set of reduced words of length $n$, admitting $v_j$ as a factor at position $i_j$  for every $1\leq j\leq m$ (if $m = 0$, then $E(\vec v,\vec\imath) = \Red$). If $n\ge 1$, we also write $E_n(\vec v,\vec\imath)$ for $E(\vec v,\vec\imath) \cap \Red_n$.

\begin{lemma}\label{lm: factors at given positions}
Let $\vec v = (v_1,\ldots,v_m)$ be  a sequence of non-empty reduced words and $\vec\imath = (i_1,\ldots,i_m)$ be a sequence of integers satisfying
$$1\leq i_1 < i_1 + |v_1| \leq i_2 < i_2 + |v_2| \leq \ldots \leq i_m + |v_m|\leq n.$$
Then
the following inequality holds:
\[
\RR_{n}\left(E(\vec v,\vec\imath)\right) \leq C^{m}  \alpha^{|v_1v_2\cdots v_m|}.
\]
In addition, if $m\ge 1$ and $\vec x = (v_1,\ldots,v_{m-1})$ and $\vec\jmath = (i_1,\ldots,i_{m-1})$, then
$$\RR_{n}(E(\vec v,\vec\imath)) \le C\alpha^{|v_m|} \RR_n(E(\vec x,\vec\jmath)).$$
\end{lemma}

\begin{proof}
The proof is by induction on $m$ and the case $m = 0$ is trivial. We now assume that $m \ge 1$ and that the inequality holds for vectors of length $m-1$. Since $(\RR_n)_n$ is prefix-heavy, we have
$$\RR_n(\Pref(uv_m)) = \RR_n(\Pref(uv_m) \mid \Pref(u))\ \RR_n(\Pref(u)) \le C\alpha^{|v_m]} \RR_n(\Pref(u))$$
for each $u$. Since $E(\vec v,\vec\imath) = \Pref(E_{i_{m}-1}(\vec x,\vec\jmath) v_{m})$, summing the previous inequality over all $u\in E_{i_{m}-1}(\vec x,\vec\jmath)$ yields
$$\RR_n(E(\vec v,\vec\imath)) \le C\alpha^{|v_m|} \RR_n(\Pref(E_{i_{m}-1}(\vec x,\vec\jmath))) = C\alpha^{|v_m|} \RR_n(E(\vec x,\vec\jmath))$$
since $n \ge i_m+|v_m|$. This concludes the proof.
\end{proof}

\begin{corollary}\label{lm: factors}
Let $v_1,\ldots,v_m$ be non-empty reduced words. The probability that a word of length $n$ admits $v_{1}, \ldots, v_m$ in that order as non-overlapping factors, is at most
$C^{m}  n^{m}  \alpha^{|v_1\cdots v_m|}$.
\end{corollary}

\begin{proof}
This is a direct consequence of Lemma~\ref{lm: factors at given positions}, summing over all possible position vectors.
\end{proof}

We now consider repeated non-overlapping occurrences of factors of a prescribed length.

\begin{lemma}\label{lm: factor repeated twice}
Let $1 \le i,j,t \le n$ be such that $i+t \le j$. The probability that a word of length $t$ occurs (resp. a word of length $t$ and its inverse occur) at positions $i$ and $j$ in a reduced word of length $n$ is at most equal to $C \alpha^t$.

The probability that a reduced word of length $n$ has two non-overlapping occurrences of a factor of length $t$ (resp. occurrences of a factor of length $t$ and its inverse) is at most equal to $Cn^2\alpha^t$.
\end{lemma}

\begin{proof}
Let $E_n(t,i,j)$ be the set of reduced words of length $n$ in which the same factor of length $t$ occurs at positions $i$ and $j$. Then $E_n(t,i,j)$ is the disjoint union of the sets $E_n((v,v),(i,j))$, where $v$ runs over $\Red_t$.  By Lemma~\ref{lm: factors at given positions}, we have
$$\RR_n(E_n(t,i,j)) = \sum_{v\in\Red_t}\RR_n(E((v,v),(i,j))) \le C\alpha^t \sum_{v\in\Red_t}\RR_n(E((v),(i))) = C\alpha^t,$$
where the last equality is due to the fact that the $E_n((v),(i))$ form a partition of $\Red_n$ when $v$ runs over $\Red_t$.

The same reasoning applied to the vectors $(v,v\inv)$ yields the analogous inequality for words containing non-overlapping occurrences of a word and its inverse.

The last part of the statement follows by summing over all possible values of $i$ and $j$.
\end{proof}

Applying Lemma~\ref{lm: factor repeated twice} with $i = 1$ and $j = n-t+1$, we get the following useful statement.

\begin{corollary}\label{lm: cyclically reduced}
For every positive integers $n,t$ such that $n > 2t$, the probability that a reduced word $u\in \Red_{n}$ is of the form $vwv\inv$, for some word $v$ of length $t$, is at most $C \alpha^{t}$.
\end{corollary}

Finally, we also estimate the probability that a word has two overlapping occurrences of a factor. Note that we do not need to consider overlapping occurrences of a word $v$ and its inverse, since a reduced word cannot overlap with its inverse.

\begin{lemma}\label{lm: overlapping factors}
Let $1 \le t < n$. The probability that a reduced word of length $n$ has overlapping occurrences of a factor of length $t$ is at most $Cnt\alpha^t$.
\end{lemma}

\begin{proof}
If a word $v$ overlaps with itself, more precisely, if $xv = vz$ for some words $x,z$ such that $0 < |x| = |z| < |v|$, then it is a classical result from combinatorics on words that $v = x^sy$ where $s = \left\lfloor \frac{|v|}{|x|} \right\rfloor \ge 1$ and $y$ is the prefix of $x$ of length $|v| - s|x|$ (see Figure~\ref{fig:schutz}).
\begin{figure}[htbp]
\centering
\includegraphics{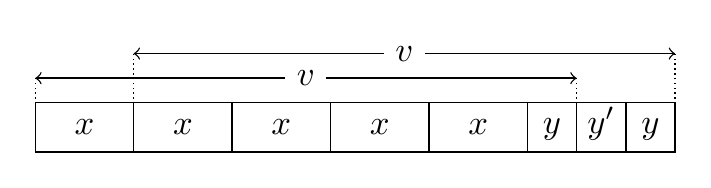} 
\caption{A classical result from  combinatorics of words: if $xv=vz$ with $0 < |x| < |v|$, then $v$ is of the form $v=x^{s}y$ for some positive integer $s$ and some prefix $y$ of $x$.\label{fig:schutz}} 
\end{figure}

It follows that, if a reduced word $u$ has (overlapping) occurrences of a factor $v$ of length $t$ at positions $i$ and $j$ ($j < i + t$), then $u$ admits a factor of the form $xv$ at position $i$, where $x$ is the prefix of $v$ of length $j-i$. Note that, once $t$ and $j-i$ are fixed, $v$ is entirely determined by $x$. Therefore this occurs with probability
$$P \le \sum_{i=1}^n \sum_{j=i+1}^{i+t-1} \sum_{x\in \Red_{j-i}} \RR_n(E((xv),(i))) = \sum_{i=1}^n \sum_{j=i+1}^{i+t-1} \sum_{x\in \Red_{j-i}} \RR_n(E((x,v),(i,j))).$$
It follows that
$$P \le \sum_{i=1}^n \sum_{j=i+1}^{i+t-1}C\alpha^t  \sum_{x\in \Red_{j-i}} \RR_n(E((x),(i))) = \sum_{i=1}^n \sum_{j=i+1}^{i+t-1} C\alpha^t \le Cnt\alpha^t.$$
by Lemma~\ref{lm: factors at given positions} and using the fact that the $E_n((x),(i))$ form a partition of $\Red_n$ when $x$ runs over $\Red_{j-i}$.
\end{proof}

\subsection{Repeated cyclic factors in random reduced words}\label {sec: repeated cyclic factors}

A word $v$ is a \emph{cyclic factor} of a word $u$ if either $u\in \tilde A^*v\tilde A^*$, or $v = v_1v_2$ and $u\in v_2\tilde A^*v_1$ -- in which case we say that $v$ is a \emph{straddling factor}. For now, we only assume that $u$ is reduced, but we will be ultimately interested in the cyclically reduced case, see Corollary~\ref{prop: repeated cyclic factors in cyclically reduced words}.

\begin{lemma}\label{lm: general factors}
Let $1 \le i, t \le n$ such that $i + t \leq n$ and let $v$ be reduced word $v$ of length $t$. Then the probability that $v$ is a cyclic factor at position $i$ of an element of $\Red_{n}$, is at most $(Cn + C^2t)\alpha^t \le 2C^2n\alpha^t$.
\end{lemma}

\begin{proof}
The probability that $v$ occurs as a (regular) factor of an element of $\Red_n$ is at most $Cn\alpha^t$ by Corollary~\ref{lm: factors}.

On the other hand, $v$ occurs as a straddling factor of $u \in \Red_n$ if $v = v_2v_1$, with $1\le \ell = |v_2| < t$ and $u\in v_1 \tilde A^* v_2$, that is, $u \in E((v_1,v_2),(1,n-\ell+1))$. By Lemma~\ref{lm: factors at given positions}, this happens with probability at most $C^2\alpha^t$. Summing over the possible values of $\ell$, we find that that $v$ occurs as a straddling factor of an element of $\Red_n$ with probability at most $C^2t\alpha^t$.

Therefore the probability that $v$ occurs in $u$ as a cyclic factor is at most $(Cn + C^2t)\alpha^t$, as announced.
\end{proof}

We now consider multiple occurrences of cyclic factors of a given length. 

\begin{lemma}\label{lm: multiple no overlap}
Let $1 \le t < n$. The probability that a reduced word of length $n$ has two non-overlapping occurrences of a cyclic factor of length $t$ (resp. an occurrence of a cyclic factor of length $t$ and its inverse), is at most $(Cn^2 + C^2nt)\alpha^t \le 2C^2n^{2}  \alpha^{t}$.
\end{lemma}

\begin{proof}
Again there are several cases, depending whether the occurrences of the word (or the word and its inverse) are both standard factors, or one of them is straddling. 

The probability that a reduced word $u\in \Red_n$ admits two non-overlapping occurrences of a (standard) factor of length $t$ (resp. occurrences of a factor of length $t$ and its inverse), is at most $C n^2 \alpha^t$ by Lemma~\ref{lm: factor repeated twice}.

We now consider the situation where $u$ has two occurrences of the same word of length $t$, one as a standard factor and one straddling: there exist integers $\ell,i$ and reduced words $v_1,v_2$ such that $0 < \ell < t$, $\ell \le i \le n-2t+\ell$, $|v_2| = \ell$, $|v_1v_2| = t$ and
$$u\in  E((v_2,v_1v_2,v_1),(1,i,n-t+\ell+1)) = E((v_2,v_1,v_2,v_1),(1,i,i+\ell,n-t+\ell+1)).$$
Applying Lemma~\ref{lm: factors at given positions} twice, we find that the probability of this event according to $\RR_n$ is at most equal to $C^2\alpha^t \RR_n(E((v_2,v_1),(1,i)))$.

Then the probability $P$ that a word in $\Red_n$ admits two non-overlapping occurrences of a factor of length $t$, one standard and one straddling, is bounded above by the sum of these values when $\ell, i, v_1, v_2$ run over all possible values:
$$P \le \sum_{\ell = 0}^t \sum_{i = \ell}^{n-2t+\ell} \sum_{v_2\in\Red_\ell} \sum_{v_1\in \Red_{t-\ell}} C^2\alpha^t \RR_n(E((v_2,v_1),(1,i))).$$
For fixed values of $\ell$ and $i$, $\Red_n$ is the disjoint union of the $E((v_2,v_1),(1,i))$ when $v_2$ runs over $\Red_\ell$ and $v_1$ runs over $\Red_{t-\ell}$. So we get
$$P \le \sum_{\ell = 0}^t \sum_{i = \ell}^{n-2t+\ell} C^2\alpha^t \le C^2nt\alpha^t.$$
Thus the probability that a reduced word of length $n$ has two non-overlapping occurrences of a word of length $t$ as cyclic factors is at most equal to $(Cn^2 + C^2nt)\alpha^t \le 2C^2n^2\alpha^t$, as announced.

Finally, we consider the situation where a factor of length $t$ and its inverse occur in $u$, with one of the occurrences straddling: that is, there exist integers $\ell,i$ and reduced words $v_1,v_2$ such that $0 < \ell < t$, $\ell \le i \le n-2t+\ell$, $|v_2| = \ell$, $|v_1v_2| = t$ and $u$ lies in
$$E((v_2,v_2\inv v_1\inv,v_1),(1,i,n-t+\ell+1)) = E((v_2,v_2\inv,v_1\inv,v_1),(1,i,i+\ell,n-t+\ell+1)).$$
As above, the probability of this event according to $\RR_n$ is at most
$$C\alpha^{t-\ell}\RR_n(E((v_2,v_2\inv,v_1\inv),(1,i,i+\ell)))$$
and the probability $P'$ that a reduced word of length $n$ has two non-overlapping occurrences of a word of length $t$ as cyclic factors, with one of them straddling, satisfies
$$P' \le  \sum_{\ell=1}^{t-1}\sum_{i = \ell}^{n-2t+\ell}\sum_{v_{2}\in\Red_{\ell}}\sum_{v_{1}\in\Red_{t-\ell}}
C\alpha^{t-\ell}\RR_n(E((v_2,v_2\inv,v_1\inv),(1,i,i+\ell))).$$
For fixed values of $\ell$, $i$ and $v_2$, $E_n((v_2,v_2\inv),(1,i))$ is the disjoint union of the $E_n((v_2,v_2\inv,v_1\inv),(1,i,i+\ell))$ when $v_1$ runs over $\Red_{t-\ell}$. Therefore we have
$$P' \le  \sum_{\ell=1}^{t-1}\sum_{i = \ell}^{n-2t+\ell}\sum_{v_{2}\in\Red_{\ell}}
C\alpha^{t-\ell}\RR_n(E((v_2,v_2\inv),(1,i))).$$
By Lemma~\ref{lm: factors at given positions} again, $\RR_n(E((v_2,v_2\inv),(1,i))) \le C\alpha^\ell \RR_n(E((v_2)(1)))$ and we get, by the same reasoning as above,
$$P' \le  \sum_{\ell=1}^{t-1}\sum_{i = \ell}^{n-2t+\ell}\sum_{v_{2}\in\Red_{\ell}}
C^2\alpha^t\RR_n(E((v_2),(1))) = \sum_{\ell=1}^{t-1}\sum_{i = \ell}^{n-2t+\ell} C^2\alpha^t \le C^2nt\alpha^t.$$
Thus the probability that a reduced word of length $n$ has an occurrence of a word of length $t$ and its inverse as a cyclic factor is, again, at most equal to $(Cn^2 + C^2nt)\alpha^t \le 2C^2n^2\alpha^t$, as announced.
\end{proof}

Finally, we give an upper bound to the probability that a reduced word has overlapping occurrences of a cyclic factor of length $t$ (observing again that a reduced word cannot have overlapping occurrences of a (cyclic) factor and its inverse).

\begin{lemma}\label{lm: multiple overlap}
Let $1 \le t < n$. The probability that a reduced word of length $n$ has overlapping occurrences of a cyclic factor of length $t$ is at most equal to  $\left(Cnt + 2C^2 t^2 \right)\alpha^{t}  \le 3\ C^2nt\alpha^{t}$. 
%
\end{lemma}

\begin{proof}
The probability that a reduced word of length $n$ has overlapping occurrences of a non-straddling factor of length $t$ is at most $Cnt\alpha^t$ by Lemma~\ref{lm: overlapping factors}.

Let us now assume that the reduced word $u \in \Red_n$ has overlapping occurrences of a cyclic factor $v$ of length $t$, with one at least of these occurrences straddling. Note that any cyclic factor of $u$ is a factor of $u^2$. Therefore, using the same arguments as for Lemma~\ref{lm: overlapping factors}, $u$ has a straddling cyclic factor of the form $xv = x^{s+1}y$, where $|x| > 0$, $y$ is a prefix of $x$ and $s \ge 1$. In particular, $v = x^sy$ and $t = s|x|+|y|$. 

It follows that $u$ is in $v_2\tilde A^*v_1$, for some $v_1, v_2$ such that $v_1v_2 = x^{s+1}y$. Denote by $\pref_\ell(z)$ and $\suff_\ell(z)$ the prefix and the suffix of length $\ell$ of a word $z$. Then there exist a cyclic conjugate $z$ of $x$ and integers $0\le h,\ell < |z| = |x|$ and $m,m'\ge 0$ such that $v_1 = \suff_h(z)z^{m'}$ and $v_2 = z^m\pref_\ell(z)$. Note that $x^{s+1}y = \suff_h(z)z^{m+m'}\pref_\ell(z)$ and
\begin{align*}
h+\ell &= |y| \pmod{|z|} \\
m+m' &= \begin{cases}s+1 &\textrm{if $h+\ell = |y|$}\\ s &\textrm{if $h+\ell = |z|+|y|$}\end{cases} \\
t+|z| &= (m+m')|z| + h + \ell.
\end{align*}
Observe also that $|y|$ is determined by $|z|$ ($|y| = t \pmod{|z|}$), that $h$ is determined by $\ell$ and $|z|$, and that $m'$ is determined by $m$, $\ell$ and $|z|$. Then
\begin{align*}
u \in \bigcup_{k = 1}^{t-1} \bigcup_{\ell = 0}^{k-1} &\bigcup_{m = 0}^{1+\lfloor\frac tk\rfloor} \bigcup_{z\in\Red_k} X_{z,m,\ell},\textrm{ where}\\
X_{z,\ell,m} &= E((z^m\pref_\ell(z),\suff_h(z)z^{m'}),(1,n - m'|z| - h +1))
\end{align*}
and $h$ and $m'$ take the values imposed by those of $k = |z|$, $\ell$ and $m$. In particular, the probability $P$ that a reduced word in $\Red_n$ has overlapping occurrences of a cyclic factor of length $t$, with at least one of these occurrences straddling, satisfies
$$P \enspace \le \enspace \sum_{k = 1}^{t-1} \sum_{\ell = 0}^{k-1}  \sum_{m = 0}^{1+\lfloor\frac tk\rfloor} \sum_{z\in\Red_k} \RR_n(X_{z,\ell,m}),$$

If $m \ge 1$, then
$$X_{z,\ell,m} = E((z,z^{m-1}\pref_\ell(z),\suff_h(z)z^{m'}),(1,|z|+1,n - m'|z| - h +1))$$
and a double application of Lemma~\ref{lm: factors at given positions} shows that 
$$\RR_n(X_{z,\ell,m}) \le C^2 \alpha^{m'|z|+h} \alpha^{(m-1)|z|+\ell} \RR_n(E((z),(1))) = C^2 \alpha^t \RR_n(E((z),(1))).$$
Summing these over $z\in \Red_k$ (with $k$, $\ell$ and $m$ fixed, $m\ge 1$), we get
$$\sum_{z\in\Red_k} \RR_n(X_{z,\ell,m}) \le \sum_{z\in\Red_k} C^2\alpha^t \RR_n(E((z),(1))) \le C^2\alpha^t,$$
since $\Red_n$ is partitioned by the $\RR_n(E((z),(1)))$ ($z\in \Red_k$).

If $m =0$ and $h+\ell = |y|$, then $m'|z| = t+|z|-|y|$ and we note that
\begin{align*}
X_{z,\ell,0} &= E((\pref_\ell(z),\suff_h(z)z^{m'}),(1,n - t - |z| + \ell +1)) \\
&\subseteq E((\pref_\ell(z),\suff_{h}(z),\suff_{|y|}(z)z^{m'-1}),(1,n - t - |z| + \ell +1,n-t+1)).
\end{align*}
By Lemma~\ref{lm: factors at given positions}, we get
$$\RR_n(X_{z,\ell,0}) \le C\alpha^t \RR_n(E((\pref_\ell(z),\suff_{h}(z)),(1,n - t - |z| + \ell +1))).$$
Summing over all $z\in \Red_k$ ($k$ and $\ell$ fixed), we get
\begin{align*}
\sum_{z\in\Red_k} \RR_n(X_{z,\ell,0}) &\le \sum_{z\in\Red_k} C\alpha^t \RR_n(E((\pref_\ell(z),\suff_{h}(z)),(1,n - t - k + \ell +1))) \\
&\le \sum_{z_1\in\Red_\ell} \sum_{z_2\in\Red_h} C\alpha^t \RR_n(E((z_1,z_2),(1,n - t - k + \ell +1))) \\
&\le C \alpha^t,
\end{align*}
since $\Red_n$ is partitioned by the $\RR_n(E((z_1,z_2),(1,n - t - k + \ell +1)))$ ($z_1\in \Red_\ell$, $z_2\in \Red_h$).

Finally, if $m = 0$ and $h+\ell = |z|+|y|$, then $m'|z| = t-|y|$. Therefore
\begin{align*}
X_{z,\ell,0} &= E((\pref_\ell(z),\suff_h(z)z^{m'}),(1,n - t - |z| + \ell +1)) \\
&= E((\pref_\ell(z),\pref_{|z|-\ell}(\suff_{h}(z)),\suff_{|y|}(z)z^{m'}),(1,n - t - |z| + \ell +1,n-t+1)).
\end{align*}
By Lemma~\ref{lm: factors at given positions}, this yields
$$\RR_n(X_{z,\ell,0}) \le C\alpha^{t} \RR_n(E((\pref_\ell(z),\pref_{|z|-\ell}(\suff_{h}(z))),(1,n - t + |z| + \ell +1))).$$
As in the previous case, summing over all $z\in \Red_k$ ($k$ and $\ell$ fixed) yields
$$\sum_{z\in\Red_k} \RR_n(X_{z,\ell,0}) \le C \alpha^{t}.$$

Then we get the following upper bound for the probability $P$:
\begin{align*}
P &\le \sum_{k = 1}^{t-1} \sum_{\ell = 0}^{k-1}  \sum_{m = 1}^{1+\lfloor\frac tk\rfloor} C^2 \alpha^t + \sum_{k = 1}^{t-1} \sum_{\ell = 0}^{k-1} C \alpha^{t} \\
& \le C^2\frac32 t(t-1) \alpha^t + C \frac12 t(t-1)\alpha^{t} \enspace\\
&\le \enspace 2 C^2 t(t-1) \alpha^{t}
 .
\end{align*}
This concludes the proof.
\end{proof}

In order to extend the results of this section to cyclically reduced words, we need an additional hypothesis, essentially stating that the probability of cyclically reduced words does not vanish. In fact, we have the following general result.

\begin{lemma}\label{lemma: conditional probability}
Let $(\RR_n)_{n\ge 0}$ be a sequence of measures satisfying $\liminf \RR_n(\calC_n) = p > 0$. Let $X$ be a subset of $\Red$. Then for each $\delta > 1$ and for every large enough $n$, the probability $\RR_n(X \mid \calC)$ that a cyclically reduced word of length $n$ is in $X$ is at most equal to $\frac\delta p \RR_n(X)$. In particular, if $X$ is exponentially (resp. super-polynomially, polynomially, simply) negligible, then so is $X \cap \calC$ in $\calC$.
\end{lemma}

\begin{proof}
By definition, $\RR_n(X \mid \calC) = \RR_n(X\cap\calC \mid \calC) = \frac{\RR_n(X\cap\calC)}{\RR_n(\calC_n)} \le \frac\delta p \RR_n(X)$, which concludes the proof.
\end{proof}

The following statement is an immediate consequence.

\begin{corollary}\label{prop: repeated cyclic factors in cyclically reduced words}
Let $(\RR_n)_{n\ge 0}$ be a prefix-heavy sequence of parameters $(C,\alpha)$, with the property that $\liminf_n \RR_n(\calC_n) = p > 0$. Then for every $\delta > 1$ and every large enough $n$, the probability that a cyclically reduced word of length $n$ has two non-overlapping occurrences of a cyclic factor of length $t$ (resp. an occurrence of a cyclic factor of length $t$ and its inverse, two overlapping occurrences of a cyclic factor of length $t$) is at most $\frac\delta p(Cn^2 + C^2nt)\alpha^t$ (resp. $\frac\delta p(Cn^2 + C^2nt)\alpha^t$, $\frac{3 \delta}{ p} C^2 nt \alpha^{t}$).
\end{corollary}

\begin{proof}
Let $X$ be the set of reduced words of length $n$ with two non-overlapping occurrences of a cyclic factor of length $t$ (resp. an occurrence of a cyclic factor of length $t$ and its inverse, two overlapping occurrences of a cyclic factor of length $t$). It suffices to apply Lemma~\ref{lemma: conditional probability} to the set $X$, and to use the results of Lemmas~\ref{lm: multiple no overlap} and~\ref{lm: multiple overlap}.
\end{proof}

\subsection{Measures on tuples of lengths and on tuples of words}

For every positive integer $k$, let $\tuple_{k}$ denote the set of $k$-tuples
of non-negative integers and $\tupleW_{k}$ denote the set of $k$-tuples
of reduced words. Let also $\tuple = \bigcup_{k} \tuple_{k}$ and $\tupleW = \bigcup_{k} \tupleW_{k}$ be the sets of all tuples of non-negative integers, and of reduced words respectively.

For a given $\vec{h}=(h_{1},\ldots, h_{k})$ of $\tupleW_{k}$,
let $\|\vec{h}\|$ be the element of $\tuple_{k}$ given by
\[
\|\vec{h}\| = \left( |h_{1}|,\ldots,|h_{k}|\right).
\]
A \emph{prefix-heavy sequence of measures on tuples of reduced words} is a sequence $(\PP_{n})_{n\geq 0}$ of measures on $\tupleW$ such that for every $\vec{h}=(h_{1},\ldots,h_{k})$ of $\tupleW$,
\[
\PP_{n}(\vec{h}) = \TT_{n}(\|\vec{h}\|) \prod_{i=1}^{k}\RR_{|h_{i}|}(h_{i}),
\]
where $(\TT_{n})_{n\geq 0}$ is a sequence of measures on $\tuple$ and $(\RR_{n})_{n\geq 0}$ is a prefix-heavy sequence of measures on $\Red$. If $(\RR_{n})_{n\geq 0}$ is prefix-heavy with parameters $(C,\alpha)$, then we say that $(\TT_{n})_{n\geq 0}$ is prefix-heavy with parameters $(C,\alpha)$.

\begin{remark}
In the definition above, to draw a tuple of words according to
$\PP_{n}$, one can first draw a tuple of lengths $(\ell_{1},\ldots,\ell_{k})$
following $\TT_{n}$, and then draw, independently for each coordinate,
an element of $\Red_{\ell_{i}}$ following $\RR_{\ell_{i}}$.
\end{remark}

\begin{example}\label{ex: uniform distribution on tuples}
Let $\nu(n)$ be an integer-valued function. The uniform distribution on the $\nu(n)$-tuples of reduced words of length exactly $n$ is a prefix-heavy sequence of measures: one needs to take $\TT_n$ to be the measure whose weight is entirely concentrated on the $\nu(n)$-tuple $(n,\ldots,n)$ and $\RR_n$ to be the uniform distribution on $\Red_n$ (see Example~\ref{ex: uniform distribution on words}).

The uniform distribution on the $\nu(n)$-tuples of reduced words of length at most $n$ is also a prefix-heavy sequence of measures. Here the support of $\TT_n$ must be restricted to the tuples $(x_1,\ldots,x_{\nu(n)})$ such that $x_i \le n$ for each $i$, with $\TT_n(x_1,\ldots,x_{\nu(n)}) = \prod_i\frac{|\Red_{x_i}|}{|\Red_{\le n}|}$.

Both can be naturally adapted to handle the uniform distribution on the $\nu(n)$-tuples of cyclically reduced words of length exactly (resp. at most) $n$.

For appropriate functions $\nu(n)$, we retrieve the few-generator and the density models discussed in Section~\ref{sec: two classical models}. We will see a more general class of examples in Section~\ref{sec: automata probabilities}.
\end{example}

\subsection{General statements}\label{sec: general statements}

If $\vec{x}\in\tuple$, we denote by $\Max(\vec{x})$ and
$\Min(\vec{x})$ the maximum and minimum element of $\vec{x}$. We also
denote by $\Nbr(\vec{x})$ the integer
$k$ such that $\vec{x}\in\tuple_{k}$.

The statistics  $\Min$, $\Max$, and $\Nbr$ are extended to tuples of words
by setting $\Min(\vec{h})=\Min(\|\vec{h}\|)$, $\Max(\vec{h})=\Max(\|\vec{h}\|)$
and  $\Nbr(\vec{h})=\Nbr(\|\vec{h}\|)$.
In the sequel we consider
sequences of probability spaces on $\tupleW$ and $\Min$, $\Max$, and
$\Nbr$ are seen as random variables.

The following statements give general sufficient conditions for a tuple to generically have the central tree property, generate a malnormal subgroup, or satisfy a small cancellation property.

\begin{proposition}\label{prop: heartsize}
Let $(\PP_{n})_{n\geq 0}$ be a prefix-heavy sequence of measures on tuples of reduced words of parameters $(C,\alpha)$. Let $f\colon \N\to\N$ such that $f(\ell) \le \frac\ell2$ for each $\ell$. If there exists a sequence $(\eta_n)_{n\geq 0}$ 
of positive real numbers such that
\begin{equation}\label{eq: heartsize}
\lim_{n\rightarrow\infty}\PP_{n}\left(\Nbr^2  \alpha^{f(\Min)} > \eta_n\right) = 0
\quad\text{and}\quad  \lim_{n\rightarrow\infty}\eta_n = 0,
\end{equation}
then a random tuple of words generically satisfies $\Lcp(\vec h) < f(\Min(\vec h))$.

If the limits in Equation~\eqref{eq: heartsize} converge polynomially (resp. super-polynomially, exponentially) fast, then $\Lcp(\vec h) < f(\Min(\vec h))$ polynomially (resp. super-polyn\-omially, exponentially) generically.
\end{proposition}

\begin{proof}
The set of all tuples $\vec h$ that fail to satisfy the inequality $\Lcp(\vec h) < f(\Min(\vec h))$ is the union $\calG_1 \cup \calG_2$ of the two following sets:
\begin{itemize}
\item the set $\calG_1$ of all tuples $\vec h = (h_1,\ldots, h_k)$ such that for some $1\le i < j \le k$, a word of length $f(\Min(\vec h))$ occurs as a prefix of $h_i$ or $h_i\inv$, and also of $h_j$ or $h_j\inv$,

\item the set $\calG_2$ of all tuples $\vec h = (h_1,\ldots, h_k)$ such that  for some $1\le i\le k$, $h_i$ and $h_i\inv$ have a common prefix of length $f(\Min(\vec h))$,
\end{itemize}
and we only need to prove that $\lim_n\P_n(\calG_1) = \lim_n\P_n(\calG_2) = 0$.

Let $k, \ell$ be positive integers and let $X_{k,\ell}$ be the set of tuples $\vec h\in \tupleW_k$ such that $\Min(\vec h) = \ell$.
If $\vec h \in X_{k,\ell}$ and $1\le i < j \le k$, then the probability that $h_i$ and $h_j$ have the same prefix of length $t = f(\ell)$ is 
$$\sum_{w\in\Red_{t}} \RR_{|h_i|}(\Pref(w))  \RR_{|h_j|}(\Pref(w)) \enspace\leq\enspace C \alpha^{t}  \sum_{w\in\Red_{t}}\RR_{|h_j|}(\Pref(w)) \enspace\leq\enspace  C \alpha^{t}.$$
Then we have $\P_n(\calG_1 \mid X_{k,\ell}) \le 4k^2C\alpha^{f(\ell)}$, or rather $\P_n(\calG_1 \mid X_{k,\ell}) \le \min(1,4k^2C\alpha^{f(\ell)})$, where the factor $k^2$ corresponds to the choice of $i$ and $j$ and the factor 4 corresponds to the possibilities that $h_i$ or $h_i\inv$, and $h_j$ or $h_j\inv$ have a common prefix of length $f(\ell)$. Therefore we have $\P_n(\calG_1 \cap X_{k,\ell}) \le \min(1,4k^2C\alpha^{f(\ell)})\ \P_n(X_{k,\ell})$

We can split the set of pairs $(k,\ell)$ into those pairs such that $k^2\alpha^{f(\ell)} > \eta_n$ and the others, for which $k^2\alpha^{f(\ell)} \le \eta_n$. Then we have
$$\P_n(\calG_1) \enspace=\enspace \sum_{k,\ell} \P_n(\calG_1 \cap X_{k,\ell}) \enspace\le\enspace \P_n(\Nbr^2  \alpha^{f(\Min)} > \eta_n) + 4C\ \eta_n,$$
which tends to 0 under the hypothesis in Equation~(\ref{eq: heartsize}).

Similarly, if $\vec h \in X_{k,\ell}$ and $i \le k$, the probability that $h_i$ and $h_i\inv$ have a common prefix of length $f(\ell)$ is at most $C\alpha^{f(\ell)}$ by Corollary~\ref{lm: cyclically reduced}. It follows that $\P_n(\calG_2 \mid X_{k,\ell}) \le \min(1,kC\alpha^{f(\ell)})$, and $\P_n(\calG_2 \cap X_{k,\ell}) \le \min(1,kC\alpha^{f(\ell)})\ \P_n(X_{k,\ell})$.

Splitting the set of pairs $(k,\ell)$ into those pairs such that $k\alpha^{f(\ell)} > \eta_n$ and those for which $k\alpha^{f(\ell)} \le \eta_n$, yields
$$\P_n(\calG_2) \enspace=\enspace \sum_{k,\ell} \P_n(\calG_2 \cap X_{k,\ell}) \enspace\le\enspace \P_n(\Nbr  \alpha^{f(\Min)} > \eta_n) + C\ \eta_n.$$
Now $\Nbr  \alpha^{f(\Min)} < \Nbr^2  \alpha^{f(\Min)}$, so $\P_n(\Nbr  \alpha^{f(\Min)} > \eta_n) \le \P_n(\Nbr^2  \alpha^{f(\Min)} > \eta_n)$. It follows that $\lim_n \P_n(\Nbr  \alpha^{f(\Min)} > \eta_n) = 0$, and hence $\lim_n\P_n(\calG_2) = 0$, which concludes the proof.
\end{proof}

\begin{theorem}[Central tree property]\label{thm: general small heart}
Let $(\PP_{n})_{n\geq 0}$ be a prefix-heavy sequence of measures on tuples of reduced words of parameters $(C,\alpha)$. If there exists a sequence $(\eta_n)_{n\geq 0}$ 
of positive real numbers such that
\begin{equation}\label{eq: small heart}
\lim_{n\rightarrow\infty}\PP_{n}\left(\Nbr^2  \alpha^{\frac\Min2} > \eta_n\right) = 0
\quad\text{and}\quad  \lim_{n\rightarrow\infty}\eta_n = 0,
\end{equation}
then a random tuple of words generically has the central tree property. In particular, such a tuple is a basis of the subgroup it generates.

If the limits in Equation~\eqref{eq: small heart} converge polynomially (resp. super-polynomially, exponentially) fast, then the central tree property holds polynomially (resp. super-polyn\-omially, exponentially) generically.
\end{theorem}

\begin{proof}
By definition, a tuple $\vec h \in \tupleW$ satisfies the central tree property if $\Lcp(\vec h) < \frac{\Min(\vec h)}{2}$, so the theorem is a direct application of Proposition~\ref{prop: heartsize} to the function $f(\ell) = \frac\ell2$, and of Proposition~\ref{fact: central tree}.
\end{proof}

\begin{theorem}[Malnormality]\label{thm: general malnormal}
Let $(\PP_{n})_{n\geq 0}$ be a prefix-heavy sequence of measures on tuples of reduced words of parameters $(C,\alpha)$. If there exists a sequence $(\eta_n)_{n\geq 0}$ 
of positive real numbers such that
\begin{equation}\label{eq: generic condition malnormal}
\lim_{n\rightarrow\infty}\PP_{n}\left(\Nbr^2 \Max^2 \alpha^{\frac\Min{8}} > \eta_n\right) = 0
\quad\text{and}\quad  \lim_{n\rightarrow\infty}\eta_n = 0,
\end{equation}
then a random tuple of words generically generates a malnormal subgroup.

If the limits in Equation~\eqref{eq: generic condition malnormal} converge polynomially (resp. super-polynomially, exponentially) fast, then malnormality holds polynomially (resp. super-polyn\-omially, exponentially) generically.
\end{theorem}

\begin{proof}
By Proposition~\ref{prop: sufficient malnormality}, a sufficient condition for a tuple $\vec h \in \tupleW$ to generate a malnormal subgroup is to have $\Lcp(\vec h) < \frac13\Min(\vec h)$, and to not have two occurrences of a word of length $\frac12(\Min(\vec h) - 3\Lcp(\vec h))$ as a factor of a word in $\vec h^\pm$. This condition is satisfied in particular if $\Lcp(\vec h) < \frac14\Min(\vec h)$ and no word of length $\frac18\Min(\vec h)$ has two occurrences as a factor of a word in $\vec h^\pm$.

Therefore the set of all tuples $\vec h$ that generate a non malnormal subgroup is contained in the union $\calG_1 \cup \calG_2 \cup \calG_3 \cup \calG_4$ of the following sets:
\begin{itemize}
\item the set $\calG_1$ of all tuples $\vec h = (h_1,\ldots, h_k)$ such that $\Lcp(\vec h) \ge \frac14\Min(\vec h)$,

\item the set $\calG_2$ of all tuples $\vec h = (h_1,\ldots, h_k)$ such that for some $1\le i < j \le k$, a word of length $\frac18\Min(\vec h)$ occurs as a factor of $h_i$, and also of $h_j$ or $h_j\inv$,

\item the set $\calG_3$ of all tuples $\vec h = (h_1,\ldots, h_k)$ such that  for some $1\le i\le k$, $h_i$ and $h_i\inv$ have a common factor of length $\frac18\Min(\vec h)$,

\item the set $\calG_4$ of all tuples $\vec h = (h_1,\ldots, h_k)$ such that  for some $1\le i\le k$, $h_i$ has at least two occurrences of a factor of length $\frac18\Min(\vec h)$,
\end{itemize}
and we want to verify that $\P_n(\calG_1)$, $\P_n(\calG_2)$, $\P_n(\calG_3)$ and $\P_n(\calG_4)$ all tend to 0 when $n$ tends to infinity.

By Proposition~\ref{prop: heartsize}, the set $\calG_1$ is negligible as soon as $\lim_n\P_n(\Nbr  \alpha^{\frac\Min4} > \eta_n) = 0$. This is true under the hypothesis in Equation~(\ref{eq: generic condition malnormal}) since $\Nbr \alpha^{\frac\Min4} < \Nbr^2\Max^2 \alpha^{\frac\Min{8}}$, and hence $\P_n(\Nbr \alpha^{\frac\Min4} > \eta_n) \le \P_n(\Nbr^2 \Max^2 \alpha^{\frac\Min{8}} > \eta_n)$.

Let now $X_{k,\ell,M}$ be the set of tuples $\vec h\in X_{k,\ell}$ such that $\Max(\vec h) = M$.
Let $1 \le i < j \le k$ and $\vec h \in X_{k,\ell,M}$. By Corollary~\ref{lm: factors}, the probability that  $h_j$ has a given factor $v$ of length $\frac\ell8$ is at most equal to $C M \alpha^{\frac\ell8}$. Summing this probability over all words $v$ which occur as a factor of $h_i$ (at most $|h_i| \le M$ such words), it follows that the probability that $h_i$ and $h_j$ have a common factor of length $t = \frac\ell8$ is at most equal to $C M^2 \alpha^{\frac\ell8}$. Summing now over the possible values of $i$ and $j$, we find that $\P_n(\calG_2 \cap X_{k,\ell,M}) \le \min(1,k^2CM^2\alpha^{\frac\ell8})\ \P_n(X_{k,\ell,M})$ and therefore, as above
$$\P_n(\calG_2) \enspace\le\enspace \P_n(\Nbr^2\Max^2 \alpha^{\frac\Min8} > \eta_n) + C\ \eta_n.$$
It follows from Equation~(\ref{eq: generic condition malnormal}) that $\calG_2$ is negligible.

By Lemma~\ref{lm: factor repeated twice}, the probability that $h_i$ and $h_i\inv$ have a common factor of length $\frac\ell8$ is at most $CM^2\alpha^{\frac\ell8}$. Summing over all choices of $i$, we find that
$$\P_n(\calG_3) \enspace\le\enspace \P_n(\Nbr \Max^2 \alpha^{\frac\Min8} > \eta_n) + C\ \eta_n.$$
Since $\Nbr \Max^2 \alpha^{\frac\Min8} < \Nbr^2 \Max^2 \alpha^{\frac\Min{8}}$, we conclude that $\calG_3$ is negligible.

Finally, we have $\P_n(\calG_4) \le \frac C8\Nbr\Max\Min\alpha^{\frac\Min{8}}$ by Lemma~\ref{lm: overlapping factors}, and hence
$$\P_n(\calG_4) \enspace\le\enspace \P_n(\Nbr\Max\Min \alpha^{\frac\Min{8}} > \eta_n) + \frac C8\ \eta_n.$$
Since $\Nbr\Max\Min \alpha^{\frac\Min{8}} < \Nbr^2\Max^2 \alpha^{\frac\Min{8}}$, it follows as above that the set $\calG_4$ is negligible.
\end{proof}

\begin{theorem}[Small cancellations property]\label{thm: general small cancellations}
Let $(\PP_{n})_{n\geq 0}$ be a prefix-heavy sequence of measures on tuples of reduced words of parameters $(C,\alpha)$, such that $\liminf_n\RR_n(\calC_n) = p > 0$. For any $\lambda\in(0,\frac12)$, if there exists a sequence $(\eta_n)_{n\geq 0}$ of positive real numbers such that
\begin{equation}\label{eq: generic condition 2}
\lim_{n\rightarrow\infty}\PP_{n}\left(\Nbr^2 \Max^2 \alpha^{\lambda\Min} > \eta_n\right) = 0
\quad\text{ and }\quad
\lim_{n\rightarrow\infty} \eta_n = 0,
\end{equation}
then  the property $C'(\lambda)$ generically holds.

If the limits in Equation~\eqref{eq: generic condition 2} converge polynomially (resp. super-polynomially, exponentially) fast, then Property $C'(\lambda)$ holds polynomially (resp. super-polyn\-omially, exponentially) generically.
\end{theorem}

\begin{proof}
A sufficient condition for a tuple of cyclically reduced words $\vec h$ to satisfy $C'(\lambda)$ is for every piece in $\vec h$ to have length less than $\lambda\Min(\vec h)$. Then the set $\calG$ of tuples that fail to satisfy $C'(\lambda)$ is contained in the union $\calG_1\cup \calG_2 \cup \calG_3 \cup \calG_4$ of the following sets:
\begin{itemize}
\item the set $\calG_1$ of all tuples of cyclically reduced words $\vec h = (h_1,\ldots, h_k)$ such that for some $1\le i < j \le k$, a word of length $\lambda\Min(\vec h)$ occurs as a factor of $h_i$, and also of $h_j$ or $h_j\inv$,

\item the set $\calG_2$ of all tuples of cyclically reduced words $\vec h = (h_1,\ldots, h_k)$ such that  for some $1\le i\le k$, $h_i$ has two non-overlapping occurrences of a factor of length $\lambda\Min(\vec h)$,

\item the set $\calG_3$ of all tuples of cyclically reduced words $\vec h = (h_1,\ldots, h_k)$ such that  for some $1\le i\le k$, $h_i$ has non-overlapping occurrences of a factor of length $\lambda\Min(\vec h)$ and its inverse,

\item the set $\calG_4$ of all tuples of cyclically reduced words $\vec h = (h_1,\ldots, h_k)$ such that  for some $1\le i\le k$, $h_i$ has overlapping occurrences of a factor of length $\lambda\Min(\vec h)$,
\end{itemize}
and we want to verify that $\P_n(\calG_1)$, $\P_n(\calG_2)$, $\P_n(\calG_3)$ and $\P_n(\calG_4)$ all tend to 0 when $n$ tends to infinity.

As in the proof of Theorem~\ref{thm: general malnormal}, we find that the probability that a tuple of reduced words $\vec h$ is such that a word of length $\lambda\Min(\vec h)$ occurs as a factor of $h_i$, and also of $h_j$ or $h_j\inv$, for some $i < j$ is at most $\P_n(\Nbr^2\Max^2 \alpha^{\lambda\Min} > \eta_n) + C\ \eta_n$. Reasoning as in the proof of Corollary~\ref{prop: repeated cyclic factors in cyclically reduced words}, it follows that, for every $\delta > 1$,
$$\P_n(\calG_1) \le \frac\delta p\left(\P_n(\Nbr^2\Max^2 \alpha^{\lambda\Min} > \eta_n) + C\ \eta_n\right),$$
and it follows from Equation~(\ref{eq: generic condition 2}) that $\calG_1$ is negligible.

Now using Corollary~\ref{prop: repeated cyclic factors in cyclically reduced words}, we show that
\begin{align*}
\P_n(\calG_2), \P_n(\calG_3) &\le \frac\delta p \left( \P_n(\Nbr(\Max^2 + \Max\Min)\alpha^{\lambda\Min} > \eta_n) + C^2\eta_n\right),\\
\P_n(\calG_4) &\le\frac\delta p \left( \P_n(\Nbr(\Max\Min + \Min^2)\alpha^{\lambda\Min} > \eta_n) + 2 C^2\eta_n\right).
\end{align*}
Since $\Nbr\Max^2$, $\Nbr\Max\Min$ and $\Nbr\Min^2$ are less than $\Nbr^2\Max^2$, the hypothesis in Equation~(\ref{eq: generic condition 2}) shows that $\calG_2$, $\calG_3$ and $\calG_4$ are negligible, and this concludes the proof.
\end{proof}

\subsection{Applications to the uniform distribution case}\label{sec: applications to uniform}

The few-generator model and the density model, based on the uniform distribution on reduced words of a given length and discussed in Section~\ref{sec: two classical models}, are both instances of a prefix-heavy sequence of measures on tuples, for which the parameter $\alpha$ is $\alpha = \frac1{2r-1}$, see Examples~\ref{ex: uniform distribution on words} and~\ref{ex: uniform distribution on tuples}. In this section, the measure $\RR_n$ is the uniform distribution on $\Red_n$.

The results of Section~\ref{sec: general statements} above allow us to retrieve many of the results in Section~\ref{sec: two classical models} --- typically the results on the small cancellation property $C'(\lambda)$ up to density $\frac\lambda2$, whether one considers tuples of cyclically reduced words of length $n$ or of length at most $n$ ---, and to expand them. In particular, we show that the results on the central tree property and malnormality in the few-generator model can be extended to the density model, and that we have a phase transition theorem for the central tree property (at density $\frac14$).

\medskip\noindent\textbf{Small cancellation properties}\enspace
Let $0 < d < 1$. In the density model, at density $d$, we choose uniformly at random a $\nu(n)$-tuple of cyclically reduced words of length $n$, with $\nu(n) = |\calC_n|^d$. In particular, for every tuple $\vec h$ of that sort, we have $\Nbr(\vec h) = \nu(n)$ and $\Max(\vec h) = \Min(\vec h) = n$.

Let $0 < \lambda < \frac12$ and for each $n$, let
$$\eta_n \enspace=\enspace \left(\frac{2r}{2r-1}\right)^{2d}n^2\ (2r-1)^{-(\lambda - 2d)n} \enspace+\enspace \left(\frac{2r}{2r-1}\right)^{d}n^2\ (2r-1)^{-(\lambda - d)n}.$$
Note that $|\calC_n| < |\Red_n| = \frac{2r}{2r-1}(2r-1)^n$. Therefore $\Nbr^2\Max^2\alpha^{\lambda\Min} <  \eta_n$ with probability 1. Now observe that $\eta_n$ converges exponentially fast to 0 when $d < \frac\lambda2$. In view of Theorem~\ref{thm: general small cancellations}, this provides a proof of part of Theorem~\ref{thm: phase transitions}~(2), namely, of the fact that, at density less than $\frac\lambda2$, Property $C'(\lambda)$ holds exponentially generically.

It is unclear whether the more difficult property, that hyperbolicity holds generically at density less than $\frac12$, can be established with the same very general tools. 

Observe that the set $\Red_{\le n}$ of reduced words of length at most $n$ has cardinality $1+\sum_{i=1}^n|\Red_n| = \frac r{r-1}(2r-1)^n - \frac1{r-1}$. By the same reasoning as above, at density less than $\frac\lambda2$, a tuple of cyclically reduced words of length at most $n$ exponentially generically has Property $C'(\lambda)$.

\medskip\noindent\textbf{Properties of subgroups}\enspace
We now return to tuples of reduced words like in the few-generator model, but with a density type assumption on the size of the tuples. 
For $0 < d < 1$, we consider $|\Red_{\le n}|^d$-tuples of reduced words of length at most $n$, and the asymptotic properties of the subgroups generated by these tuples. For such tuples $\vec h$, we have $\Nbr(\vec h) \le \big(\frac{r}{r-1}\big)^d(2r-1)^{dn}$ and $\Max(\vec h) = n$.

In addition, for every $0 < \mu < 1$, Proposition~\ref{prop: few generators basic} shows that $\Min(\vec h) > \mu n$, exponentially generically.

We first establish the central tree property.

\begin{proposition}\label{prop: uniform ctp at density}
Let $0 < d < \frac14$. At density $d$, a tuple of reduced words of length at most $n$ chosen uniformly at random, exponentially generically has the central tree property, and in particular it is a basis of the subgroup it generates.

If $d>\frac14$, then at density $d$ the central tree property exponentially generically does not hold.
\end{proposition}

\begin{proof}
For a fixed $\mu < 1$, the following inequality holds exponentially generically:
$$\Nbr^2\alpha^{\frac\Min2} \le \left(\frac{r}{r-1}\right)^{2d} (2r-1)^{-(\frac\mu2-2d)n}.$$
At every density $d < \frac14$, one can choose $\mu < 1$ such that $\frac\mu2-2d > 0$ (say, $\mu = \frac{1+4d}2$). For such a value of $\mu$, $\eta_n = \left(\frac{r}{r-1}\right)^{2d} (2r-1)^{-(\frac\mu2-2d)n}$ converges exponentially fast to 0 and, in view of Theorem~\ref{thm: general small heart}, this proves the first part of the proposition.

If $d>\frac14$, let $d'$ be such that $\frac14 < d' < \min(\frac12,d)$. By the classical Birthday Paradox\footnote{If $E$ is a set of size $M$ and $x$ is a uniform random tuple of $E^{m}$, the probability that the coordinates of $x$ are pairwise distinct is $(1-\frac1{M})(1-\frac2{M})\cdots(1-\frac{m-1}{M})$, which is
at most $\exp(-\frac{m(m-1)}{2M})$ by direct calculations.},  exponentially generically two words of the tuple share a prefix of length $2d'n$. This prove the second part of the proposition.
\end{proof}

Along the same lines, we also prove the following result.

\begin{proposition}\label{prop: uniform malnormal at density}
Let $0 < d < \frac1{16}$. At density $d$, a tuple of reduced words of length at most $n$ chosen uniformly at random, exponentially generically generates a malnormal subgroup.
\end{proposition}

\begin{proof}
For a fixed $\mu < 1$, we have
$$\Nbr^2\Max^2\alpha^{\frac\Min8} \le \left(\frac{r}{r-1}\right)^{2d} n^2 (2r-1)^{-(\frac\mu8-2d)n},$$
exponentially generically.

If $d < \frac1{16}$, one can choose $\mu < 1$ such that $\frac\mu8 -2d > 0$ (say, $\mu = \frac{1+16d}2$), and we conclude as above, letting
$$\eta_n =  \left(\frac{r}{r-1}\right)^{2d} n^2 (2r-1)^{-(\frac\mu8-2d)n}$$
and using Theorem~\ref{thm: general malnormal}.
\end{proof}

\begin{remark}
Propositions~\ref{prop: uniform ctp at density} and~\ref{prop: uniform malnormal at density} above generalize Corollary~\ref{cor: few generators}~(1) and~(2), from the few generator case to an exponential number of generators --- up to density $\frac14$ and $\frac1{16}$, respectively (see Proposition~\ref{prop: inherited properties}).
\end{remark}

Proposition~\ref{prop: uniform ctp at density} can actually be radically refined if the tuples have less than exponential size and if we drop the requirement of exponential genericity.

\begin{proposition}\label{prop: very small central trees}
Let $f$ be an unbounded non-decreasing integer function. Let $k > 1$ be a fixed integer. Then a $k$-tuple $\vec h$ of reduced words of length at most $n$ chosen uniformly at random, generically has the central tree property, with $\Lcp(\vec h) \le f(n)$.

Let $c, c' > 0$ such that $c' \log(2r-1) > 2c$. Then an $n^c$-tuple $\vec h$ of reduced words of length at most $n$ chosen uniformly at random, generically has the central tree property, with $\Lcp(\vec h) \le c'\log n$.
\end{proposition}

\begin{proof}
If $k$ is a fixed integer, then as in the proof of Proposition~\ref{prop: uniform ctp at density}, we find that, for each $\mu < 1$, $\Nbr^2\alpha^{f(\Min)}$ is generically less than or equal to $\eta_n = k^2(2r-1)^{-f(\mu n)}$, which tends to $0$. This concludes the proof on the size of the central tree of random $k$-tuples by Proposition~\ref{prop: heartsize}.

If we now consider $n^c$-tuples, we find that, for each $\mu < 1$, $\Nbr^2\alpha^{c'\log(\mu n))}$ is generically less than or equal to $\eta_n = n^{2c}(2r-1)^{-c'\log n} = n^{-(c'\log(2r-1)-2c)}$, which tends to 0. By Proposition~\ref{prop: heartsize} again, this concludes the proof.
\end{proof}

\section{Markovian automata}\label{sec: automata probabilities}

We now switch from the very general settings of the previous section to a specific and computable
way to define prefix-heavy sequences of measures on reduced words. 

We introduce Markovian automata (Section~\ref{sec: Markovian definition}) which determine prefix-heavy sequences of measures under a simple and natural non-triviality assumption. These automata are a form of hidden Markov chain, and when they have a classical ergodicity property, then cyclically reduced words have asymptotically positive density. We are then able to generalize the results of Section~\ref{sec: applications to uniform} about central tree property and malnormality.

In the last part of the section, we give a generalization of Theorem~\ref{thm: phase transitions}~(2) and~(3) on small cancellation and the degeneracy of a finite presentation.

\subsection{Definition and examples}\label{sec: Markovian definition}

A \emph{Markovian automaton}\footnote{%
  This notion is different from the two notions of probabilistic
  automata, introduced by Rabin \cite{1963:Rabin} and Segala and Lynch
  \cite{1995:SegalaLynch}, respectively.}
$\calA$ consists of 
\begin{itemize}
\item a deterministic transition system $(Q, \cdot)$ on alphabet $X$,
  where $Q$ is a finite non-empty set called the \emph{state set}, and
  for each $q\in Q$, $x\in X$, $q\cdot x \in Q$ or $q\cdot x$ is
  undefined;
  
  \item an initial probability vector $\gamma_0\in [0,1]^Q$, that is, a positive vector such that $\sum_{q\in Q}\gamma_0(q) = 1$;
  
  \item for each $p\in Q$, a probability vector $(\gamma(p,x))_{x\in X} \in [0,1]^X$, such that $\gamma(p,x) = 0$ if and only if $p\cdot x$ is undefined.
\end{itemize}

If $u = x_0\cdots x_n \in X^*$ ($n\ge 0$), we write $\gamma(q,u) =
\gamma(q,x_0)   \gamma(q\cdot x_0,x_1) \cdots \gamma(q\cdot(x_0\cdots
x_{n-1}),x_n)$. We let $\gamma(q,u)=1$ if $u$ is the empty word. We also
write $\gamma_0(u) = \sum_{q\in Q}\gamma_0(q)  \gamma(q,u)$.

Markovian automata are very similar to hidden Markov chain models, except that symbols are output on transitions instead of on states. We will discuss this further in Section~\ref{sec: local} below. Markovian automata can be considered as more intuitive since sets of words (languages) are naturally described by automata.

We observe that, for each $n\ge 0$, $\sum_{|u| = n}\gamma(u) = 1$. Thus $\gamma$ determines a probability measure $\RR_n$ on the set of elements of $X^*$ of length $n$: if $|u| = n$, then $\RR_n(u) = \gamma(u)$.

In the sequel, we consider only Markovian automata on alphabet $\tilde A$, where only reduced words have non-zero probability. More precisely, the \emph{support} of a Markovian automaton $\calA$ is the set of words that can be read in $\calA$, starting from a state $q$ such that $\gamma_0(q) \ne 0$, that is, the set of all words $u$ such that $\gamma(u) \ne 0$: we assume that our Markovian automata are such that their support is contained in $\Red$.

\begin{example}\label{ex:uniform} 
\emph{Uniform distribution on reduced words of length $n$.}
It is immediately verified that the following Markovian automaton yields the uniform distribution on reduced words of each possible length. The state set is $Q = \tilde A$. For each $a\in \tilde A$, there is an $a$-labeled transition from every state except $a\inv$, ending in state $a$. All these transitions have the same probability, namely $\frac1{2r-1}$, and the initial probability vector is uniform as well, with each coordinate equal to $\frac1{2r}$.

One can also tweak these probabilities, to favor certain letters over others, or to favor positive letters (the letters in $A$) over negative letters.
\end{example}

\begin{example}\label{ex:psl2} \emph{Distributions on rational subsets of $F(A)$.}
The support of a Markovian automaton $\calA$ is always rational and closed under taking prefixes, but it does not have to be equal to the set of all reduced words. We can
consider a rational subset $L$ of $F(A)$, or rather a deterministic
transition system reading only reduced words, and impose probabilistic
weights on its transitions to form a Markovian automaton. The
resulting distribution gives non-zero weights only to prefixes of
elements of $L$.

\begin{figure}[htbp]
\begin{center}
\null\hfill
$(\calA)$\includegraphics[scale=1]{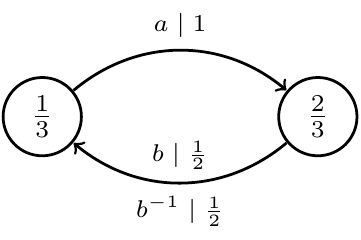}
\hfill
\includegraphics[scale=1]{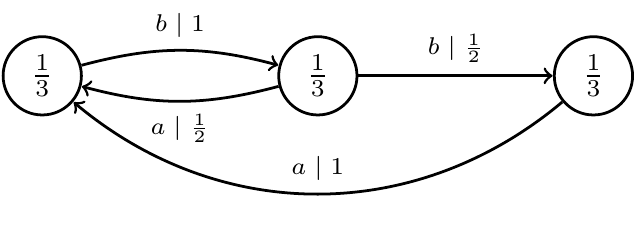}$(\calA')$
\hfill\null
\caption{Markovian automata $\calA$ and $\calA'$.\label{fig: 2 automata}}
\end{center}
\end{figure}
Figure~\ref{fig: 2 automata} represents two such automata (transitions are labeled by a letter and a probability, and each state is decorated with the corresponding initial probability), which are related with the modular group, $PSL(2,\Z) = \langle a,b \mid a^2, b^3\rangle$.

The support of the distribution defined by automaton $\calA$ is the
set of words over alphabet $\{a,b,b\inv\}$ without occurrences of the
factors $a^2$, $b^2$, $(b\inv)^2$, $bb\inv$ and $b\inv b$, and the
support of the distribution defined by $\calA'$ consists of the words
on alphabet $\{a,b\}$, without occurrences of $a^2$ or $b^3$. Both are
regular sets of unique representatives of the elements of $PSL(2,\Z)$:
the first is the set of geodesics of $PSL(2,\Z)$, and also the set of
Dehn-reduced words with respect to the given presentation of that
group; the second is a set of quasi-geodesics of $PSL(2,\Z)$.  Notice
that the distribution produced by $\calA'$ is not uniform on words of
length $n$ of its support.
\end{example}

Example~\ref{ex:uniform} shows that the sequence $(\RR_n)_n$ of uniform measures on reduced words, discussed in Sections~\ref{sec: two classical models} and~\ref{sec: applications to uniform} can be specified by a Markovian automaton. We also know that this sequence is prefix-heavy (Example~\ref{ex: uniform distribution on words}). This is a general fact, under mild assumptions on the Markovian automaton.

\begin{proposition}\label{prop: Markovian is prefix-heavy}
Let $\calA$ be a Markovian automaton and let $(\RR_n)_n$ the sequence of probability measures it determines. If $\calA$ does not have a cycle with probability 1, then $(\RR_n)_n$ is a prefix-heavy sequence of measures, with computable parameters $(C,\alpha)$.
\end{proposition}

\begin{proof}
Let $\ell$ be the maximum length of an elementary cycle (one that does
not visit twice the same state) and let $\delta$ be the maximum value
of $\gamma(q,\kappa)$ where $\kappa$ is an elementary cycle at state
$q$. Under our hypothesis, $\delta < 1$.

Every cycle $\kappa$ can be represented as a composition of at least
${|\kappa|}/{\ell}$ elementary cycles (here, the composition takes the
form of a sequence of insertions of a cycle in another). Consequently
$\gamma(q,\kappa) \le \delta^{\frac{|\kappa|}{\ell}}$. Finally, every
path can be seen as a product of cycles and at most $|Q|$
individual edges. So, if $u$ is a word and $q\in Q$, then $\gamma(q,u) \le
\delta^{\frac{|u|-|Q|}{\ell}}$, that is $\gamma(q,u) \le C\alpha^{|u|}$ where $C = \delta^{\frac{-|Q|}{\ell}}$ and $\alpha = \delta^{\frac1\ell}$.

Let $u,v$ be reduced words such that $uv$ is reduced and let $n\ge |uv|$. We have
\begin{align*}
\RR_n(\Pref(uv)) = \gamma_0(uv) &= \sum_{p\in Q} \gamma_0(p)\gamma(p,u)\gamma(p\cdot u,v) \\
&\le \left(\sum_{p\in Q} \gamma_0(p)\gamma(p,u)\right)\ C\alpha^{|v|}\\
&= \gamma_0(u)\ C\alpha^{|v|} = \RR_n(\Pref(u))\ C\alpha^{|v|},
\end{align*}
and hence $\RR_n(\Pref(uv) \mid \Pref(u)) \le C\alpha^{|v|}$, which concludes the proof.
\eop

\begin{remark}
The parameters $C$ and $\alpha$ described in the proof of Proposition~\ref{prop: Markovian is prefix-heavy} may be far from optimal. If $\beta < 1$ is a uniform bound on the probabilities of the transitions of $\calA$, then $\gamma_0(v), \gamma(q,v) \le \beta^{|v|}$ for each word $v$, and the computation in the proof above shows that $\RR_n(\Pref(uv) \mid \Pref(u)) \le \beta^{|v|}$. We will see in Section~\ref{sec: local} that we can be more precise under additional hypotheses.
\end{remark}

Now let $\calA$ be a Markovian automaton without a probability 1 cycle, such that the sequence of probability measures it induces is prefix-heavy with parameters $(C,\alpha)$. If $0 < d < 1$, we say that a tuple $\vec h$ of reduced words of length at most (resp. exactly) $n$ is chosen at random according to $\calA$, \emph{at $\alpha$-density $d$} if $\vec h$ consists of $\alpha^{-dn}$ words. Observe that this generalizes the concept discussed in Section~\ref{sec: density model} and~\ref{sec: applications to uniform}.

With the same proofs as in Section~\ref{sec: applications to uniform}, we have the following generalization of Propositions~\ref{prop: uniform ctp at density} and~\ref{prop: uniform malnormal at density} related to central tree property and malnormality.

\begin{corollary}\label{cor: Markovian ctp}
Let $\calA$ be a Markovian automaton without a probability 1 cycle, such that the induced sequence of probability measures is prefix-heavy with parameters $(C,\alpha)$.  Then a tuple of reduced words of length at most $n$ chosen at random according to $\calA$, at $\alpha$-density $d < \frac14$, exponentially generically has the central tree property.

At $\alpha$-density $d < \frac1{16}$, it exponentially generically generates a malnormal subgroup.
\end{corollary}

\subsection{Irreducible Markovian automata and coincidence probability}\label{sec: irreducible}\label{sec: local}

An $(n,n)$-matrix $M$ is said to be \emph{irreducible} if it has non-negative coefficients and, for every $i,j \le n$, there exists $s\ge 1$ such that $M^s(i,j) > 0$. Equivalently, this means that $M$ is not similar to a block upper-triangular matrix. We record the following general property of irreducible matrices.

\begin{lemma}\label{lm:eigenvalue}
Let $M$ be an irreducible matrix. Then its spectral radius $\rho$ is a (positive) eigenvalue with a positive eingenvector. In particular, there exist positive vectors $\vec v_{\min}$ and $\vec v_{\max}$ such that, componentwise,
$$\rho^n \vec v_{\min} \enspace\leq\enspace M^{n}  \vec 1 \enspace\leq\enspace  \rho^n \vec v_{\max}\quad\textrm{for all $n > 0$}$$
where $\vec 1$ is the vector whose coordinates are all equal to 1. Moreover, there exist $c_{\min}, c_{\max} > 0$ such that
$$c_{\min} \rho^n \enspace\leq\enspace \vec 1^t  M^{n}   \vec 1 \enspace\leq\enspace c_{\max} \rho^n\quad\textrm{for all $n > 0$.}$$
\end{lemma}

\begin{proof}
We refer the reader to \cite[chap. 13, vol. 2]{1959:Gantmacher} for a comprehensive
presentation of the properties of irreducible matrices and in particular for the Perron-Frobenius theorem, which establishes that the spectral radius of $M$ is an eigenvalue with a positive eigenvector: let $\vec v_0$ be such an eigenvector, and let $\vec v_{\min}$ (resp. $\vec v_{\max}$) be appropriate multiples of $\vec v_0$ with all coefficients less than $1$ (resp. greater than $1$). Then we have, componentwise, $\rho^n \vec v_{\min} = M^n\vec v_{\min} \leq \vec  M^{n} \vec 1 \leq M^n \vec v_{\max} =\rho^n \vec v_{\max}$.

Let $c_{\min}$ (resp. $c_{\max}$) be the sum of the coefficients of $\vec v_{\min}$ (resp. $\vec v_{\max}$). Then, summing over all components of $M^n\vec v_{\min}$ and $M^n\vec v_{\max}$, we get $c_{\min} \rho^n \leq \vec 1^t  M^{n}   \vec 1 \leq c_{\max} \rho^n$.
\end{proof}

Going back to automata, we note that a \emph{Markov chain} can be naturally associated with a Markovian automaton: if $\calA$
is a Markovian automaton on alphabet $\tilde A$, with state set $Q$, we
define the Markov chain $M(\calA)$ on $Q$ as follows: its transition
matrix is given by $M(p,q) = \sum_{a\in \tilde A\textrm{ s.t. }p\cdot a = q}
\gamma(p,a)$ for all $p,q\in Q$, and its initial vector is $\gamma_0$.

We say that the Markov chain $M(\calA)$ (or, by extension, the Markovian automaton $\calA$), is \emph{irreducible} if this transition matrix is irreducible, which is
equivalent to the strong connectedness of $\calA$.  We note that, in that case, if $\calA$ does not consist of a simple cycle, then $\calA$ does not have a cycle of probability 1. In view of Proposition~\ref{prop: Markovian is prefix-heavy}, this implies that the sequence of probability measures determined by $\calA$ is prefix-heavy. We will see below (Proposition~\ref{prop: parameter if irreducible}) that we can give a precise evaluation of the parameters of this sequence.

To this end, we introduce the notion of local Markovian automata, where labels can be read on states instead of edges.

More precisely a Markovian automaton is \emph{local} if all the incoming transitions into a given state are labeled by the same letter: for all states $p,q$ and letters $a, b$, if $p\cdot a = q\cdot b$ then $a=b$. If $\A$ is a Markovian
automaton, let $\A'$ denote the local Markovian automaton obtained as
follows.
\begin{itemize}
\item its set of states is $Q' = \{(q,a)\in Q\times \tilde{A}\mid \exists p\in Q,\ p\cdot a = q\}$;
\item its transition function $\star$ is given by $(p,a)\star b = (q,b)$ if $p\cdot b = q$;
\item its initial probability vector $\gamma'_{0}$ is given by
\[
\gamma_{0}'\big((p,a)\big) =
\begin{cases}
\gamma_{0}(p) & \text{if }a\text{ is the least label of the transitions into }p \\
0 &\text{otherwise}
\end{cases}
\]
 (we fix an arbitrary order on $\tilde{A}$)
\item its transition probability vectors are given by $\gamma'\big((p,a),b\big) = \gamma(p,b)$.
\end{itemize}
\begin{figure}[htbp]
\centering
\includegraphics[scale=1.2]{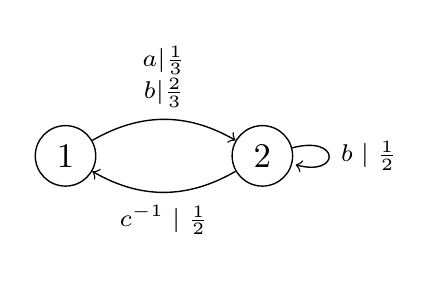}
\qquad
\includegraphics[scale=1.2]{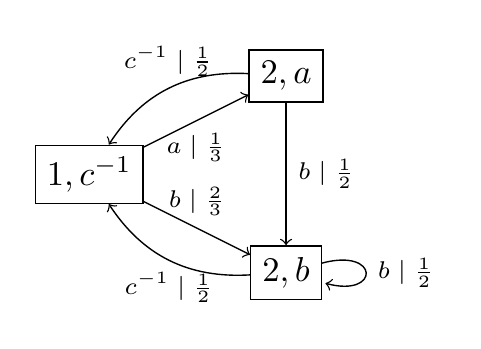}
\caption{A Markovian automaton and its associated local automaton.\label{fig:local}}
\end{figure}
\begin{proposition}\label{pro: local}
Let $\A$ be a Markovian automaton. Then the associated local Markovian automaton $\A'$ assigns
the same probability as $\A$ to every reduced word. Moreover, if $\A$ is irreducible, then so is
$\A'$.
\end{proposition}

\begin{proof}
The first part of the statement follows directly from the definition, by a simple induction on the length of the words: indeed, we retrieve a path in $\A$ by forgetting
the second coordinate on the states of $\A'$; and every path of $\A$ starting at some state $q$, can be lifted uniquely to a path in $\A'$ starting at any vertex of the form $(q,a)$ of $\A'$.

Assume that $\A$ is irreducible and let $(p,a)$ and $(q,b)$ be states of $\A'$. By definition of $\A'$, there exists a state $q'$ of $\A$ such that $q'\cdot b = q$. Moreover, since $\A$ is irreducible, there exists
a path from $p$ to $q'$ in $\A$, say $p\xrightarrow{a_{1}}q_{1}\xrightarrow{a_{2}}\ldots\xrightarrow{a_t} q'$.
Then
$$(p,a)\xrightarrow{a_{1}}(q_{1},a_{1})\xrightarrow{a_{2}}\ldots\xrightarrow{a_t} (q',a_t) \xrightarrow{b} (q',b)$$
is a path in $\A'$ from $(p,a)$ to $(q,b)$, so $\A'$ is irreducible as well.
\end{proof}

If $\A$ is a Markovian automaton, we denote by $\MM_{\A}$ (or just $\MM$ when there is no ambiguity) the stochastic matrix associated with
its local automaton $\A'$:
\[
\MM\big((p,a),(q,b)\big) =
\begin{cases}
\gamma'\big((p,a),b\big) = \gamma(p,b) & \text{if $p\cdot b = q$}\\
0 &\text{otherwise.}
\end{cases}
\]
We also denote by $\MM_{[2]}$ and $\MM_{[3]}$ the matrices defined by
\begin{align*}
\MM_{[2]}\big((p,a),(q,b)\big) &= \Big(\MM\big((p,a),(q,b)\big)\Big)^2 \textrm{ and }\\
\MM_{[3]}\big((p,a),(q,b)\big) &=  \Big(\MM\big((p,a),(q,b)\big)\Big)^3,
\end{align*}
and by $\coincidence$ and $\alpha_{[3]}$ the largest eigenvalue of $\MM_{[2]}$ and $\MM_{[3]}$, respectively. The value $\coincidence$ is called the \emph{coincidence probability} of $\A$, and it will play an important role in the sequel.

Observe that if $\A$ is local, then $\A'$ is equal to $\A$, up to the name of the states. We are interested in local automata for the following properties.
\begin{lemma}\label{fact: local automaton}
Let $\A$ be a local Markovian automaton. Then the following holds%
\begin{itemize}
\item for all states $p,q$ there is at most one transition from $p$ to $q$;
\item two paths starting from the same state are labeled by the same word if and only if they go through the same states in the same order;

\item for every $\ell\ge 0$, we have $\MM^{\ell}(p,q) = \sum_{\substack{u\in\Red_{\ell}, p\cdot u = q}} \gamma(p,u)$, $\MM_{[2]}^{\ell}(p,q) = \sum_{\substack{u\in\Red_{\ell}, p\cdot u = q}} \gamma(p,u)^2$ and $\MM_{[3]}^{\ell}(p,q) = \sum_{\substack{u\in\Red_{\ell}, p\cdot u = q}} \gamma(p,u)^3$.
\end{itemize}
\end{lemma}
We can now give an upper bound for the parameters of the sequence of probability measures determined by an irreducible Markovian automaton.

\begin{proposition}\label{prop: parameter if irreducible}
Let $\calA$ be an irreducible Markovian automaton with coincidence probability $\coincidence$, and let $(\RR_n)_n$ be the sequence of probability measures it determines. If $\calA$ does not consist of a single cycle, then there exists a constant $C > 0$ such that $(\RR_n)_n$ is prefix-heavy with parameters $(C,\coincidence^{1/2})$.
\end{proposition}

\begin{proof}
Let $v$ be a reduced word of length $\ell$ and let $q\in Q$ be a state of $\calA$. By Lemma~\ref{fact: local automaton}, we have
$$\gamma(q,v) = \sqrt{\gamma(q,v)^2} \le \sqrt{\MM_{[2]}^\ell(q,q\cdot v)} \le \sqrt{\vec 1^t \MM_{[2]}^\ell \vec 1}.$$
Lemma~\ref{lm:eigenvalue} then shows that there exists $C > 0$ such that $\gamma(q,v) \le C \coincidence^{\frac\ell2}$. We can now conclude as in the proof of Proposition~\ref{prop: Markovian is prefix-heavy}.
\end{proof}

This yields the following refinement of Corollary~\ref{cor: Markovian ctp}.

\begin{corollary}\label{cor: irreducible Markovian ctp}
Let $\calA$ be a Markovian automaton without a probability 1 cycle and with coincidence probability $\coincidence$. Then a tuple of reduced words of length at most $n$ chosen at random according to $\calA$, at $\coincidence$-density $d < \frac18$ (resp. $d < \frac1{32}$), exponentially generically has the central tree property (resp. generates a malnormal subgroup).
\end{corollary}

\subsection{Ergodic Markovian automata}\label{sec: ergodic}

If the Markovian automaton $\calA$ is irreducible and if, in addition,
for all large enough $n$, $M(\calA)^n(q,q) > 0$ for each $q\in Q$, we
say that $\calA$ (resp. $M(\calA)$) is \emph{ergodic}. This is
equivalent to stating that $\calA$ has a collection of loops of
relatively prime lengths, or also that all large enough integral powers
of $M(\calA)$ have only positive coefficients.
If $\calA$ is ergodic, we can apply a classical theorem on Markov chains, which states that there exists a \emph{stationary vector} $\tilde\gamma$ such that the
distribution defined by $\calA$ converges to that stationary vector
exponentially fast (see \cite[Thm 4.9]{2009:LevinPeresWilmer}). In the
vocabulary of Markovian automata, this yields the following theorem.

If $u\in \tilde A^*$ has length $n$, let $Q_n^p(u) = p\cdot u$ be the state of $\calA$
reached after reading the word $u$ starting at state $p$. We treat $Q_n^p$ as a random variable.

\begin{theorem}\label{thm ergodique}
  Let $\calA$ be an ergodic Markovian automaton on
  alphabet $\tilde A$, with state set $Q$ ($|Q|\geq 2$). For each $q\in Q$, the limit $\lim_{n\to\infty} \RR_n[Q_n^p = q]$ exists, and if we denote it by $\tilde\gamma(q)$, then $\tilde\gamma$ is a probability vector (called the \emph{stationary vector}). In addition, there exist $K>0$
  and $0 < c <1$, such that $|\RR_n[Q_n^p = q] - \tilde\gamma(q)| < K
  c^n$ for all $n$ large enough.
\end{theorem}

\begin{remark}
  The constant $c$ in Theorem~\ref{thm ergodique} is the maximal
  modulus of the non-1 eigenvalues of $M(\calA)$.
\end{remark}

\begin{example}
The Markovian automaton discussed in Example~\ref{ex:uniform}, relative to the uniform distribution on reduced words of length $n$, is ergodic. Its stationary vector $\tilde\gamma$ is equal to $\gamma_0$ ($\tilde\gamma(q) = \frac1{2r}$ for every state $q$), and the constant $c$ is $\frac1{2r-1}$.

On the other hand, the Markovian automaton $\calA$ in Example~\ref{ex:psl2} is irreducible but not ergodic (loops have even lengths), and it does not have a stationary vector.
\end{example}

We use Theorem~\ref{thm ergodique} to show that, under a very mild additional hypothesis, an ergodic Markovian automaton yields a prefix-heavy sequence of measures $(\RR_n)_n$ such that $\liminf\RR_n(\calC) > 0$.

\begin{proposition}\label{Markovian, liminf C}
Let $\calA$ be an ergodic Markovian automaton, with initial vector $\gamma_0$ and stationary vector $\tilde\gamma$ and let $(\RR_n)_n$ be the sequence of measures it induces on reduced words. If $\sum_{a\in\tilde A}\gamma_0(a)\tilde\gamma(a\inv) \ne 1$, then $\liminf\RR_n(\calC) > 0$.
\end{proposition}

Observe that the sum $\sum_{a\in\tilde A}\gamma_0(a)\tilde\gamma(a\inv)$ is less than 1, since we are dealing with probability vectors, unless there exists a (necessarily single) letter $a$ such that $\gamma_0(a) = \tilde\gamma(a\inv) = 1$.

\begin{proof}
The set $\calC$ of cyclically reduced words is the complement in $\Red$ of the disjoint union of the sets $a\tilde A^*a\inv$ ($a\in \tilde A$). Now we have
\begin{align*}
\RR_n(a\tilde A^*a\inv) &=  \sum_{p\in Q} \gamma_0(p)\gamma(p,a) \left(\sum_{|u| = n-2}\gamma(p\cdot a,u)\gamma(p\cdot(au),a\inv)\right) \\
&= \sum_{p\in Q} \gamma_0(p)\gamma(p,a) \left(\sum_{q\in Q} \RR_n(Q_{n-2}^{p\cdot a} = q) \gamma(q,a\inv)\right) \\
&= \sum_{p\in Q} \gamma_0(p)\gamma(p,a) \left(\sum_{q\in Q} (\tilde\gamma(q) + \epsilon(q,n)) \gamma(q,a\inv)\right),
\end{align*}
where $|\epsilon(q,n)| \le Kc^{n-2}$, with $K$ and $c$ given by Theorem~\ref{thm ergodique}. Then we have
$$\RR_n(a\tilde A^*a\inv) =  \gamma_0(a)\tilde\gamma(a\inv) + \gamma_0(a)\left(\sum_{q\in Q}\epsilon(q,n)\gamma(q,a\inv)\right)$$
and $\lim\RR_n(a\tilde A^*a\inv) = \gamma_0(a)\tilde\gamma(a\inv)$. It follows that
$$\lim\RR_n(\calC) = 1 - \sum_{a\in\tilde A}\gamma_0(a)\tilde\gamma(a\inv),$$
thus concluding the proof.
\eop

Proceeding as in Section~\ref{sec: applications to uniform}, we can use Proposition~\ref{Markovian, liminf C}, Corollary~\ref{prop: repeated cyclic factors in cyclically reduced words} and the results of Section~\ref{sec: general statements}, to generalize part of Theorem~\ref{thm: phase transitions}~(2), and show that, up to $\coincidence$-density $\frac\lambda4$, a tuple of cyclically reduced words of length at most $n$ chosen at random according to $\calA$, exponentially generically satisfies the small cancellation property $C'(\lambda)$. We will now see (Theorem~\ref{thm: phase}) that we can improve this bound, and go up to $\coincidence$-density $\frac\lambda2$.

\subsection{Phase transitions for the Markovian model}\label{sec: phase transitions}

We can now state a phase transition theorem, which generalizes parts of Theorem~\ref{thm: phase transitions}. Let us say that an ergodic Markovian automaton is \emph{non-degenerate} if its initial distribution $\gamma_0$ and its stationary vector $\tilde\gamma$ satisfy $\sum_{a\in\tilde A}\gamma_0(a)\tilde\gamma(a\inv) \ne 1$.

\begin{theorem}\label{thm: phase}
Let $\A$ be a non-degenerate ergodic Markovian automaton with coincidence probability $\coincidence$. Let $0 < d < 1$ and let $G$ be the group presented by a tuple $\vec h$ of cyclically reduced words of length $n$, chosen independently and at random according to $\A$, at $\coincidence$-density $d$. Then we have the following phase transitions:
\begin{itemize}
\item if $0 < \lambda < \frac12$ and $0 < d < \frac\lambda2$, then exponentially generically $\vec h$ satisfies the small cancellation property $C'(\lambda)$; if $\lambda = \frac16$, then $G$ is generically infinite and hyperbolic;
\item if $d > \frac\lambda2$ then exponentially generically $\vec h$ does not satisfy the small cancellation property $C'(\lambda)$;
\item if $d > \frac12$ then exponentially generically $G$ is degenerated in a sense that is made precise in Proposition~\ref{prop: degenerate subgroups}, and which implies that $G$ is a free group or the free product of a free group with $\Z/2\Z$.
\end{itemize} 
\end{theorem}
The rest of the paper is devoted to the proof of Theorem~\ref{thm: phase}. The first statement is established in Proposition~\ref{pro: Markovian small cancel}, while  the second and third statements are  proved respectively in Propositions~\ref{prop: no small cancellation}  and~\ref{prop: degenerate subgroups}.

\subsection{Long common factors at low density}\label{sec: long factors}

In this section we estimate the probability that random words share a long common factor. More precisely, we show the following statement, the first part of Theorem~\ref{thm: phase}.

\begin{proposition}\label{pro: Markovian small cancel}
Let $\A$ be a non-degenerate ergodic Markovian automaton with coincidence probability $\coincidence$. Let $\lambda\in(0,\frac12)$ and let $d\in(0,\frac{\lambda}{2})$. A tuple of cyclically
reduced words of length $n$ taken independently and randomly according to $\A$, at $\coincidence$-density $d$, exponentially generically satisfies the small cancellation property $C'(\lambda)$.
\end{proposition}

The structure of the proof of Proposition~\ref{pro: Markovian small cancel} resembles that of the proof of Theorem~\ref{thm: general small cancellations}, and requires the consideration of several cases. This is the object of the rest of Section~\ref{sec: long factors}.

To this end, we introduce additional notation: let $\vec\gamma_q(n)$ be the vector of coordinates $\gamma(q,u)$ when $u$ ranges over $\Red_n$ in lexicographic order, and let $\|\vec\gamma_q(n)\|_{k} = (\sum_{u\in\Red_{n}}\gamma(q,u)^{k})^{1/k}$ be the $\ell_{k}$-norm of this vector.
We start with an elementary result.

\begin{lemma}\label{fact: probability of a factor}
Let $\A$ be a Markovian automaton, let $0 < i,\ell < n$ be integers, and let $u \in \Red_\ell$. The probability $\gothicp$ that $u$ occurs as a cyclic factor at position $i$ in a reduced word of length $n$ is bounded above by
\[ 
\left \{
  \begin{array}{ll}
   \sum_{q\in Q} \gamma(q,u) & \text{ if  } i \le n-\ell+1 \\
 \sum_{q,q'\in Q} \gamma(q,u_1)\gamma(q',u_2)  & \text{ if  } i > n-\ell+1 \text{ and } u = u_1u_2 \text{ with } |u_1| = n-i+1
  \end{array}
\right.
\]
%
\end{lemma}

\begin{proof}
If $i \le n-\ell+1$, then $\gothicp = \RR_n(\tilde A^{i-1}u\tilde A^{n - \ell - i +1})$ is equal to
\begin{align*}
\sum_{p\in Q}\gamma_0(p)\sum_{w\in\Red_{i-1}}\gamma(p,w)\gamma(p\cdot w,u) &= \sum_{p\in Q}\gamma_0(p)\sum_{q\in Q}\sum_{\substack{w\in\Red_{i-1}\\p\cdot w = q}}\gamma(p,w)\gamma(q,u) \\
&= \sum_{p\in Q}\gamma_0(p) \sum_{q\in Q} \RR_{i-1}[Q_{i-1}^{p} = q]\gamma(q,u) \\
&\le \sum_{p,q\in Q} \gamma_0(p)\gamma(q,u) = \sum_{q\in Q} \gamma(q,u).
\end{align*}
If $i > n-\ell+1$ and $u = u_1u_2$ with $|u_1| = n-i+1$, then
\begin{align*}
\gothicp = \RR_n(u_2\tilde A^{n-\ell}u_1) &= \sum_{q'\in Q} \gamma_0(q')\gamma(q',u_2) \sum_{w\in\Red_{n-\ell}} \gamma(q'\cdot u_2,w) \gamma(q'\cdot u_2w, u_1)\\
&= \sum_{q'\in Q} \gamma_0(q')\gamma(q',u_2) \sum_{q\in Q}\sum_{\substack{w\in\Red_{n-\ell}\\ q'\cdot u_2w = q}} \gamma(q'\cdot u_2,w) \gamma(q, u_1) \\
&= \sum_{q'\in Q} \gamma_0(q')\gamma(q',u_2) \sum_{q\in Q} \RR_{n-\ell}[Q_{n-\ell}^{q'\cdot u_2} = q] \gamma(q, u_1) \\
&\le \sum_{q,q'\in Q} \gamma(q,u_1) \gamma(q', u_2),
\end{align*}
which concludes the proof.
\end{proof}

\begin{proposition}\label{lm: Markovian factor}
Let $\A$ be an irreducible Markovian automaton with coincidence probability $\coincidence$. Let $n$, $\ell$, $i$ and $j$ be positive integers such that $\ell\leq n$ and $i,j\leq n$.  Denote by $L(n,\ell,i,j)$ the probability
that two reduced words of length $n$ share a common cyclic factor of length $\ell$ at positions respectively  $i$ and $j$. Then there exists a positive constant $K$ such that
\[
L(n,\ell,i,j) \leq K  \coincidence^{\ell}.
\]
\end{proposition}

\begin{proof} 
Without loss of generality (see Proposition~\ref{pro: local}), we may assume that $\A$ is local.
The proof is based on a case study.

\noindent {\it Case 1: $i, j \le n-\ell+1$.} Using Lemma~\ref{fact: probability of a factor}, we have
\[
L(n,\ell,i,j) \le \sum_{p,q\in Q} \sum_{u\in\Red_{\ell}} \gamma(p,u)  \gamma(q,u).
\]
By a repeated application of the Cauchy-Schwarz inequality, we get
\begin{equation}\label{eq: upper L}
L(n,\ell,i,j) \enspace\leq\enspace \sum_{p,q\in Q}  \|\vec\gamma_{p}(\ell)\|_{2} \|\vec\gamma_{q}(\ell)\|_{2} \enspace\le\enspace \sum_{q\in Q}\|\vec\gamma_q(\ell)\|_2^2.
\end{equation}
Now, in view of Lemma~\ref{fact: local automaton} and since $\A$ is local, we have
\begin{equation}\label{eq: sum of coefficients}
\sum_{q\in Q}\|\vec\gamma_q(\ell)\|_2^2 =  \sum_{q\in Q}\sum_{\substack{u\in\Red_{\ell}}} \gamma(q,u)^{2} = \sum_{p\in Q}\sum_{q\in Q}\sum_{\substack{u\in\Red_{\ell}\\p\cdot u = q}} \gamma(p,u)^{2} = \vec 1^{t}  \MM^{\ell}_{[2]}   \vec 1.
\end{equation}
Since  $\MM$ is irreducible, Lemma~\ref{lm:eigenvalue} shows that there exists a positive constant $K > 0$ such that, for $\ell$ large enough, we have
$$L(n,\ell,i,j) \enspace \le \enspace \sum_{q\in Q}\|\vec\gamma_q(\ell)\|_2^2 \enspace=\enspace \vec 1^{t}  \MM^{\ell}_{[2]}   \vec 1 \enspace\leq\enspace K \coincidence^{\ell},$$
which concludes the proof of the statement in that case.

\noindent {\it Case 2: $i > n-\ell+1$ and $j \le n-\ell+1$.} (The case where $i \le n-\ell+1$ and $j > n-\ell+1$ is symmetrical.) Let $k = n-i+1$ (so $1\le k < \ell$). By Lemma~\ref{fact: probability of a factor}, we have
\begin{align*}
L(n,\ell,i,j) &\le \sum_{\substack{u_1\in\Red_k \\ u_2\in \Red_{\ell-k}}}\sum_{p,p',q\in Q}  \gamma(p, u_1)\gamma(p',u_{2}) \gamma(q,u_1u_2) \\
&\le \sum_{\substack{u_1\in\Red_k \\ u_2\in \Red_{\ell-k}}}\sum_{p,p',q,q'\in Q}  \gamma(p, u_1)\gamma(p',u_{2}) \gamma(q,u_1)\gamma(q',u_2) \\
&\le \left(\sum_{u_1\in\Red_k}\sum_{p,q\in Q}  \gamma(p, u_1) \gamma(q,u_1)\right)\ \left(\sum_{u_2\in\Red_{\ell-k}}\sum_{p',q'\in Q}  \gamma(p', u_2) \gamma(q',u_2)\right).
\end{align*}
By Cauchy-Schwarz, it follows that
\begin{align*}
L(n,\ell,i,j) &\le \left(\sum_{p,q\in Q}  \|\vec\gamma_p(k)\|_2 \|\vec\gamma_{q}(k)\|_2\right)\ \left(\sum_{p',q'\in Q} \|\vec\gamma_{p'}(\ell-k)\|_2 \|\vec\gamma_{q'}(\ell-k)\|_2\right) \\
&\le \left(\sum_{q\in Q}  \|\vec\gamma_q(k)\|_2^2\right)\ \left(\sum_{q\in Q} \|\vec\gamma_q(\ell-k)\|_2^2\right) \\
&\le \left(\vec 1^t\ \MM_{[2]}^k\ \vec 1\right)\ \left(\vec 1^t\ \MM_{[2]}^{\ell-k}\ \vec 1\right)\textrm{ by Equation~\eqref{eq: sum of coefficients}}.
\end{align*}
By Lemma~\ref{lm:eigenvalue}, there exists a constant $K_1$ such that these two factors are bounded above, respectively, by $K_1 \coincidence^k$ and $K_1 \coincidence^{\ell-k}$. Therefore
$$L(n,\ell,i,j) \le K_1^2 \coincidence^\ell$$
as announced.

\noindent {\it Case 3: $i,j > n-\ell+1$.} Without loss of generality, we may assume that $i < j$, and we let $k = n-j+1$ and $k' = \ell-(n-i+1)$. Then a word $u$ of length $\ell$ occurs as a cyclic factor in two reduced words $w_1$ and $w_2$ of length $n$, at positions $i$ and $j$ respectively, if $u = u_1u_2u_3$ with $|u_1| = k$, $|u_2| = j-i$ and $|u_3| = k'$, and if $w_1 \in u_3\tilde A^{n-\ell}u_1u_2$ and $w_2 \in u_2u_3\tilde A^{n-\ell}u_1$. Then we have
\begin{align*}
L(n,\ell,i,j) &\le \sum_{\substack{u_1\in\Red_k\\u_2\in\Red_{j-i}\\u_3\in\Red_{k'}}} \sum_{\substack{p,p'\in Q \\q,q''\in Q}} \gamma(q,u_1u_2)\gamma(q'',u_3)\ \gamma(p,u_1)\gamma(p',u_2u_3)\\
&\le \sum_{\substack{u_1\in\Red_k\\u_2\in\Red_{j-i}\\u_3\in\Red_{k'}}} \sum_{\substack{p,p',p''\in Q \\q,q',q''\in Q}} \gamma(q,u_1)\gamma(q',u_2)\gamma(q'',u_3)\ \gamma(p,u_1)\gamma(p',u_2)\gamma(p'',u_3) \\
&\le \sum_{\substack{u_1\in\Red_k\\p',q\in Q}}\gamma(q,u_1)\gamma(p',u_1) \sum_{\substack{u_2\in\Red_{j-i}\\p,q''\in Q}}\gamma(q'',u_2)\gamma(p,u_2) \sum_{\substack{u_3\in\Red_{k'}\\p'',q'\in Q}}\gamma(q',u_3)\gamma(p'',u_3).
\end{align*}
By the Cauchy-Schwarz inequality, $L(n,\ell,i,j)$ is at most equal to
$$\sum_{p,q\in Q}\|\vec\gamma_p(k)\|_2 \|\vec\gamma_q(k)\|_2\ \sum_{p,q\in Q}\|\vec\gamma_p(j-i)\|_2 \|\vec\gamma_q(j-i)\|_2\ \sum_{p,q\in Q}\|\vec\gamma_p(k')\|_2 \|\vec\gamma_q(k')\|_2$$
and hence to
$$\sum_{q\in Q}\|\vec\gamma_q(k)\|_2^2\ \sum_{q\in Q}\|\vec\gamma_q(j-i)\|_2^2\ \sum_{q\in Q}\|\vec\gamma_q(k')\|_2^2.$$
Lemma~\ref{lm:eigenvalue} shows that these three factors are bounded above, respectively, by $K_1 \coincidence^k$, $K_1 \coincidence^{j-i}$ and $K_1 \coincidence^{k'}$ for some constant $K_1$. Therefore
$$L(n,\ell,i,j) \le K_1^3 \coincidence^{k+j-i+k'} = K_1^3 \coincidence^\ell,$$
as announced.
\end{proof}

\begin{proposition}\label{lm: Markovian factor and inverse}
Let $\A$ be an irreducible Markovian automaton with coincidence probability $\coincidence$.
Denote by $L^{(2)}(n,\ell,i,j)$ the probability for two reduced words of length $n$ to have an occurrence of a factor of length $\ell$ in the first word at position $i$, and an occurrence of its inverse in the second word, at position $j$, with $\ell\leq n$ and $i,j\leq n-\ell+1$. Then there exists a positive
constant $K$ such that
\[
L^{(2)}(n,\ell,i,j) \leq K  \coincidence^{\ell}.
\]
\end{proposition}

\begin{proof} 
The proof follows the same steps as that of Proposition~\ref{lm: Markovian factor}. In the first case ($i,j\le n-\ell+1$), Lemma~\ref{fact: probability of a factor} shows that
\[
L^{(2)}(n,\ell,i,j) \le \sum_{p,q\in Q}  \sum_{u\in\Red_{\ell}}\gamma(p,u)  \gamma(q,u\inv).
\]
Since the set of reduced words of length $\ell$  and the set of their inverses are equal, we get, by the Cauchy-Schwarz inequality,
\[
L^{(2)}(n,\ell,i,j) \leq \sum_{p,q\in Q}  \|\vec\gamma_{p}(\ell)\|_{2} \|\vec\gamma_{q}(\ell)\|_{2},
\]
and the proof proceeds as in the corresponding case of Lemma~\ref{lm: Markovian factor}.

In the second case ($i > n-\ell+1$ and $j \le n-\ell+1$), if $k = n-i+1$, then we have
\begin{align*}
L^{(2)}(n,\ell,i,j) &\le \sum_{\substack{u_1\in\Red_k \\ u_2\in \Red_{\ell-k}}}\sum_{p,p',q\in Q}  \gamma(p, u_1)\gamma(p',u_{2}) \gamma(q,u_2\inv u_1\inv) \\
&\le \sum_{\substack{u_1\in\Red_k \\ u_2\in \Red_{\ell-k}}}\sum_{p,p',q,q'\in Q}  \gamma(p, u_1)\gamma(p',u_{2}) \gamma(q,u_2\inv)\gamma(q',u_1\inv) \\
&\le \Big(\sum_{u_1\in\Red_k}\sum_{p,q'\in Q}  \gamma(p, u_1) \gamma(q',u_1\inv)\Big)\Big(\sum_{u_2\in\Red_{\ell-k}}\sum_{p',q\in Q}  \gamma(p', u_2) \gamma(q,u_2\inv)\Big)
\end{align*}
and as in the previous case, the proof proceeds as in Lemma~\ref{lm: Markovian factor}.

The situation is a little more complex in the last case ($i,j > n-\ell+1$). Without loss of generality, we may assume that $i < j$. With the same notation as in the proof of Lemma~\ref{lm: Markovian factor}, we distinguish two cases. If $|u_3| < |u_2|$ (that is, $\ell-k < k'$, or $\ell+i+j < 2n+2$), we let $u_2 = u'_2u''_2$ with $|u'_2| = |u_3|$. Then $w_1 \in u_3\tilde A^{n-\ell}u_1u'2u''_2$ and $w_2 \in {u'_2}\inv u_1\inv\tilde A^{n-\ell}u_3\inv{u''_2}\inv$ and, as in the previous proof, we find that $L^{(2)}(n,\ell,i,j)$ is at most equal to the sum of the
$$\gamma(p,u_1)\gamma(q,u_1\inv) \gamma(p',u'_2)\gamma(q',{u'_2}\inv) \gamma(p'',u''_2)\gamma(q'',{u''_2}\inv) \gamma(p''',u_3)\gamma(q''',u_3\inv)$$
with $u_1\in\Red_{j-i}, u'_2\in\Red_{\ell-k}, u''_2\in\Red_{k'-(\ell-k)}, u_3\in\Red_{\ell-k}$, and $p,p',p'',p''',q,q',q'',q'''$ are states in $Q$. The proof then proceeds as before, with multiple applications of the Cauchy-Schwarz inequality.

The case where $|u_3| \ge |u_2|$ (that is, $\ell+i+j \ge 2n+2$) is handled in the same fashion.
\end{proof}

\begin{corollary}\label{cor: two cyclically reduced words multiple occurrences}
Let $\A$ be a non-degenerated ergodic Markovian automaton with coincidence probability $\coincidence$. Let $n, \ell, i, j$ be positive integers such that $\ell\le n$ and $i,j\le n$. There exists a constant $K > 0$ such that the probability $\gothicp$ that two cyclically reduced words of length $n$ have occurrences of the same word of length $\ell$ (resp. of a word of length $\ell$ and its inverse) as cyclic factors at positions respectively $i$ and $j$, satisfies $\gothicp \le K\coincidence^\ell$.
\end{corollary}

\begin{proof}
The hypothesis on $\A$ guarantees that $\liminf\RR_n(\calC) = p > 0$ by Proposition~\ref{Markovian, liminf C}. Our statement then follows from Propositions~\ref{lm: Markovian factor} and~\ref{lm: Markovian factor and inverse}, in view of Lemma~\ref{lemma: conditional probability}.
\end{proof}

We now consider the case of multiple occurrences of a length $\ell$ cyclic factor (or of such a word and its inverse) within a single reduced word.

\begin{proposition}\label{lm: two occurrences}
Let $\A$ be a non-degenerate ergodic Markovian automaton with coincidence probability $\coincidence$. There exists a constant $K > 0$ such that the probability that a cyclically reduced word of length $n$ has two occurrences of a length $\ell$ word as cyclic factors, or occurrences of a length $\ell$ word and its inverse as cyclic vactors, is at most $K  \ell^2   n^{2} \coincidence^{\ell/2}$.
\end{proposition}

\begin{proof}
By Proposition~\ref{prop: parameter if irreducible}, the sequence $(\R_n)_n$ induced by $\A$ is prefix-heavy with parameters $(C,\coincidence^{1/2})$ for some $C$. The result then follows from Corollary~\ref{prop: repeated cyclic factors in cyclically reduced words}.
\end{proof}

\noindent We can now proceed with the {\bf proof of Proposition~\ref{pro: Markovian small cancel}.}
Let $N =  \coincidence^{-dn}$. An $N$-tuple of cyclically reduced words which fails to satisfy $C'(\lambda)$, must satisfy one of the following conditions: either two words in the tuple have occurrences of the same cyclic factor of length $\ell = \lambda n$ or occurrences of such a word and its inverse; or a word in the tuple has two occurrences of the same cyclic factor of length $\ell$ or occurrences of such a word and its inverse.

By Corollary~\ref{cor: two cyclically reduced words multiple occurrences}, the first event occurs with probability at most
$$K\binom{N}{2}n^2\coincidence^{\ell} \le K n^2 \coincidence^{(\lambda-2d)n}$$
for some $K > 0$. By Proposition~\ref{lm: two occurrences}, the second event occurs with probability at most
$$KN\ell^2n^2\coincidence^{\frac\ell2} \le K n^4 \coincidence^{(\frac\lambda2 - d)n},$$
for some $K >0$. Thus both events occur with probabilities that vanish exponentially fast, and this concludes the proof of Proposition~\ref{pro: Markovian small cancel}.

\subsection{Long common prefixes at high density}\label{sec: long prefixes}

In this section, we establish the following propositions corresponding
respectively to the second and third statement of Theorem~\ref{thm:
  phase}.

\begin{proposition}\label{prop: no small cancellation}
Let $\A$ be a non-degenerate ergodic Markovian automaton with coincidence probability $\coincidence$. Let $\lambda\in(0,\frac12)$ and let $d\in(\frac{\lambda}{2},1)$. A tuple of cyclically
reduced words of length $n$ taken independently and randomly according to $\A$, at density $d$, generically does not satisfy the small cancellation property $C'(\lambda)$.
\end{proposition}

\begin{proposition}\label{prop: degenerate subgroups}
Let $\calA$ be a non-degenerate ergodic Markovian automaton with coincidence probability $\coincidence$. Let $E$ be the set of letters of $\tilde A$ which label a transition in $\calA$ and let $D = A \setminus (E \cup E\inv)$. Let $d > \frac12$ and $N \ge \coincidence^{-dn}$, and let $G$ be a group presented by an $N$-tuple of cyclically reduced words chosen independently at random according to $\calA$.

If $E \cap E\inv = \emptyset$, then $G = F(|D|+1)$ exponentially generically.

If $E \cap E\inv \ne \emptyset$, then exponentially generically $G = F(D) \ast \ZZ/2\ZZ$ (if $n$ is even) or $G = F(D)$ (if $n$ is odd).
\end{proposition}

Both proofs rely heavily on the methodology introduced by Szpankowski~\cite{1991:Szpankowski} to study the typical heigth of a random trie. We first establish simple lower and upper bounds for words to share a common prefix (Lemmas~\ref{lm: prefixes} and~\ref{lm: three prefixes}).

\begin{lemma}\label{lm: prefixes}
Let $\A$ be an irreducible Markovian automaton with coincidence probability $\coincidence$. Let $P(n,\ell)$ (\emph{resp.}  $P'(n,\ell)$) be the probability that two reduced (resp. cyclically reduced) words of length $n$ share a common prefix
of length $\ell$. There exists a constant $K > 0$ such that $P(n,\ell) \geq K  \coincidence^{\ell}$.

If $\A$ is non-degenerate and ergodic and $t$ is large enough for all the coefficients of $\MM^{t}$ to be positive, then $K$ can be chosen such that $P'(n,\ell) \geq K  \coincidence^{\ell}$ when $n \ge \ell+t+1$.
\end{lemma}

\begin{proof}
Let $p$ be a state such that $\gamma_{0}(p)>0$. To establish the announced lower bounds, we only need to consider the words that can be read from state $p$. More precisely, when considering reduced words, we have
\[
P(n,\ell) \enspace\geq\enspace \gamma_{0}(p)^{2}  \sum_{u\in\Red_{\ell}}\gamma(p,u)^{2}.
\]
We observe that $\sum_{u\in\Red_{\ell}}\gamma(p,u)^{2}$ is the $p$-component of $\MM^\ell_{[2]} \vec 1$, and by Lemma~\ref{lm:eigenvalue}, it is greater than or equal to $\beta \coincidence^\ell$, where $\beta$ is the minimal component of $\vec v_{\min}$ (in the notation of Lemma~\ref{lm:eigenvalue}). This completes the proof of the statement concerning $P(n,\ell)$.

We now consider cyclically reduced words, under the hypothesis that $\A$ is non-degenerate and ergodic. Let $t$ be such that all the coefficients of $\MM^t$ are positive, let $\bar p_{\min}$ be the least coefficient of this matrix, and let $p_{\min}$ be the least positive coefficient of $\MM$. Finally, let $\gothicp = \liminf \RR_n(\calC)$, which is positive by Proposition~\ref{Markovian, liminf C}. Let $X$ (resp. $X_p$) be the set of pairs of cyclically reduced words of length $n$ that have a common prefix of length $\ell$ (resp. which can be read from state $p$). We note that
$$P'(n,\ell) \enspace=\enspace \frac{\RR_n(X)}{\RR_n(\calC)^2} \enspace\ge\enspace \frac1{\gothicp^2}\RR_n(X) \enspace\ge\enspace \frac1{\gothicp^2}\RR_n(X_p),$$
so we only need to find a lower bound for $\RR_n(X_p)$.

Suppose that $n \ge \ell + t + 1$. Then $X_p$ contains the set of pairs of reduced words of the form $(uu_1u'_1a,uu_2u'_2a)$ which can be read from $p$, where $a$ is the first letter of $u$, and $u'_1$ and $u'_2$ are words of length $t$ such that $p\cdot (uu_1u'_1) = p\cdot (uu_2u'_2) = p$. Since these words start and end with the same letters, they are guaranteed to be cyclically reduced. Thus we have
\[
\RR_n(X_p) \enspace\ge\enspace \gamma_{0}(p)^{2} \sum_{u\in\Red_{\ell}} \gamma(p,u)^{2}\  p_{\min}^{2}\ \bar p_{\min}^{2} \enspace\ge\enspace \beta\ \gamma_0(p)^2\ p_{\min}^{2}\ \bar p_{\min}^{2}\ \coincidence^\ell,
\]
and this concludes the proof.
\end{proof}

\begin{lemma}\label{lm: three prefixes}
Let $\A$ be an irreducible Markovian automaton with coincidence probability $\coincidence$. There exists a constant $K > 0$ such that the probability that three reduced words share the same prefix of length $\ell$ is at most $K \alpha_{[3]}^{\ell}$.

If $\A$ is non-degenerate and ergodic, the same holds for triples of cyclically reduced words.
\end{lemma}

\begin{proof}
The probability $\gothicp(u)$ that three reduced words have a common prefix $u$ is
\[
\gothicp(u) = \sum_{p_{1},p_{2},p_{3}\in Q}\gamma_{0}(p_{1})\ \gamma_{0}(p_{2})\ \gamma_{0}(p_{3}) \ \gamma(p_{1},u)\ \gamma(p_{2},u) \ \gamma(p_{3},u).
\]
The probability we are interested in is obtained by summing over all $u\in\Red_{\ell}$. It is bounded above by
\[
\sum_{p_{1},p_{2},p_{3}\in Q}\ \sum_{u\in\Red_{\ell}} \gamma(p_{1},u)\ \gamma(p_{2},u) \ \gamma(p_{3},u).
\]
By the H\"older and Cauchy-Schwarz inequalities, we have
\begin{align*}
\sum_{u\in\Red_{\ell}}\gamma(p_{1},u) \ &\gamma(p_{2},u) \ \gamma(p_{3},u) \\
& \leq \left(\sum_{u\in\Red_{\ell}}\gamma(p_{1},u)^{3}\right)^{\frac13}\ \left(\sum_{u\in\Red_{\ell}}\gamma(p_{2},u)^{\frac32} \ \gamma(p_{3},u)^{\frac32}\right)^{\frac23} \\
&\leq \left(\sum_{u\in\Red_{\ell}}\gamma(p_{1},u)^{3}\right)^{\frac13}\ \left(\sum_{u\in\Red_{\ell}}\gamma(p_{2},u)^{3}\right)^{\frac13}\ \left(\sum_{u\in\Red_{\ell}}\gamma(p_{3},u)^{3}\right)^{\frac13}.
\end{align*}
Moreover, we have
\[
\sum_{p\in Q}\sum_{u\in\Red_{\ell}}\gamma(p_{1},u)^{3} = \vec 1^t\ \MM^{\ell}_{[3]} \ \vec 1.
\]
We now get the announced result using Lemma~\ref{lm:eigenvalue}, Lemma~\ref{fact: local automaton} and the spectral properties of $\MM^{\ell}_{[3]}$. The generalisation to cyclically reduced words follows from Lemma~\ref{lemma: conditional probability}.
\end{proof}

We now build on the previous lemmas to show that, exponentially generically, large tuples of cyclically reduced words contain pairs of words with a common prefix of a prescribed length.

\begin{proposition}\label{pro: collisions}
Let $\A$ be an irreducible Markovian automaton with coincidence probability $\coincidence$. Let $(\ell_n)_n$ be an unbounded, monotonous sequence of positive integers such that  $\ell_n \le n$ for each $n$, and let $d > \frac12$. Then an $\coincidence^{-d\ell_n}$-tuple of reduced words of length $n$ drawn randomly according to $\A$ generically contains two words with the same prefix of length $\ell_n$.

If $\A$ is non-degenerate and ergodic, the same holds for $\coincidence^{-d\ell_n}$-tuples of cyclically reduced words.
\end{proposition}

\begin{proof}
We use the so-called \emph{second moment method}, as developed in~\cite{1991:Szpankowski}, and we introduce the following notation to this end.  Since the results of~\cite{1991:Szpankowski} are established for right-infinite words, we need to considered such words first; the result on words of length $n$ directly follows by truncation.
A right-infinite reduced word is an element $u$ of $\tilde{A}^{\mathbb{N}}$ such that for every $i\in\mathbb{N}$, $u_{i}\neq u_{i+1}\inv$. We define the probability distribution $\RR_{\infty}$ on right-infinite words induced by the Markovian automaton 
$\A$ by first setting $\RR_{\infty}(\Pref_{\infty}(u)) = \gamma(u)$, where 
$\Pref_{\infty}(u)$ is the set of right-infinite reduced words $w$ such that the finite reduced word $u$ is a prefix of $w$. The probability is then extended to 
the $\sigma$-algebra generated by the $\Pref_{\infty}(u)$, when $u$ ranges over
all finite reduced words (see~\cite{2009:ValleeClementFillFlajolet} for more details on this kind of constructions). 
Let $N = \coincidence^{-d\ell_n}$ and consider an $N$-tuple $\vec h = (h_i)_{1\le i\le N}$ of right-infinite reduced words, independently and randomly generated according to $\A$.

For $1\leq i < j \leq N$, let $X_{i,j}$ be the random variable computing the length of the longest common prefix of $h_i$ and $h_j$. We want to show that, exponentially generically,
$$\max_{1\leq i < j \leq N}X_{i,j} \geq \ell_n.$$
Let us relabel the random variables $X_{i,j}$ ($i\ne j$) as $Y_1,\ldots, Y_m$, with $m=\binom{N}2$ and, say, $Y_1 = X_{1,2}$. We are therefore computing the maximum of $m$ random variables, which are identically distributed but not independent. Fortunately, they behave almost
as if they were independent, as we will see. 

Let $d'$ be such that $\frac12 < d' < d$ and for each $m\ge 1$, let
$$r_{m} = \log_{\coincidence^{-2d'}}(m) = \frac{\log\binom{N}{2}}{\log \coincidence^{-2d'}} \sim \frac{\log\coincidence^{-2d\ell_n}}{\log \coincidence^{-2d'}} = \frac{d\ell_n}{d'}.$$
In particular, $r_m$ is asymptotically greater than $\ell_n$, and we only need to show that
\begin{equation}\label{eq:lim P(Y)=1}
\lim_{n\rightarrow\infty} \RR_\infty\left(\max_{k\in[m]}Y_{k}\geq r_{m}\right) = 1.
\end{equation}
Let $\nu(r_{m})$ denote the quantity
$$\nu(r_{m}) = \sum_{k=2}^{m}\frac{\RR_\infty(Y_{1}\geq r_{m},Y_{k}\geq r_{m})}{m\,\RR_\infty(Y_{1}\geq r_{m})^{2}}.$$
We use Lemma~3 in \cite{1991:Szpankowski}, which states that the desired equation~\eqref{eq:lim P(Y)=1} holds if
$$\lim_{n\rightarrow\infty}m\,\RR_\infty(Y_{1} > r_{m}) = +\infty\textrm{ and }\lim_{n\rightarrow\infty}\nu(r_{m}) = 1.$$
We now proceed with the proof of these two equalities.
By Lemma~\ref{lm: prefixes}, we have $\RR_\infty(Y_{1}\geq r_{m})\geq K\,\coincidence^{r_{m}}$. Then
\begin{align*}
\log\left(m\RR_\infty(Y_{1}\geq r_{m})\right) & \geq \log m + \log K + r_{m}\log\coincidence \\
& = r_{m}\log(\coincidence^{-2d'})+ \log K + r_{m}\log\coincidence \\
& = r_{m}\log(\coincidence^{1-2d'}) + \log K,
\end{align*}
which tends to $+\infty$, since $1-2d'<0$ and $\coincidence<1$. Therefore,
$$\lim_{n\rightarrow\infty}m\,\RR_\infty(Y_{1}\geq r_{m}) = +\infty.$$

Let us now consider $\nu(r_{m})$. Note that, if the $Y_i$ were independent random variables, we would have $\nu(r_m) = \frac{m-1}m$, which tends to 1 when $n$ tends to $\infty$.

Observe that if $2<i<j\leq N$, then $X_{1,2}$ and $X_{i,j}$ are independant and identically distributed, so
$$\RR_\infty(X_{1,2}\geq r_{m}, X_{i,j}\geq r_{m}) =  \RR_\infty(X_{1,2}\geq r_{m})\,\RR_\infty(X_{i,j}\geq r_{m})  = \RR_\infty(Y_{1}\geq r_{m})^{2}.$$
Also, since $h_1$ and $h_2$ are drawn independently, we have $\RR_\infty(X_{1,2}\geq r_{m},X_{1,k}\geq r_{m}) = \RR_\infty(X_{1,2}\geq r_{m},X_{2,k}\geq r_{m})$ for each $k\geq 3$. Therefore
$$\nu(r_{m}) = 2\ \sum_{k=3}^{N}\frac{\RR_\infty(X_{1,2}\geq r_{m},X_{1,k}\geq r_{m})}{m\,\RR_\infty(Y_{1}\geq r_{m})^{2}} + \binom{N-2}{2}\frac1m.$$
Since $m=\binom{N}2$, we have $\lim_n\binom{N-2}{2}\frac1m =1$. Moreover, the joint probability $\RR_\infty(X_{1,2}\geq r_{m},X_{1,k}\geq r_{m})$ is exactly the probability that three random reduced words share a common prefix of length
$r_{m}$: by Lemma~\ref{lm: three prefixes}, this is at most equal to $K\,\alpha_{[3]}^{r_{m}}$ for some constant $K > 0$. Together with Lemma~\ref{lm: prefixes}, this yields
$$\sum_{k=3}^{N}\frac{\RR_\infty(X_{1,2}\geq r_{m},X_{1,k}\geq r_{m})}{m\,\RR_\infty(Y_{1}\geq r_{m})^{2}} \leq \frac{K'}{N}\left(\frac{\alpha_{[3]}}{\coincidence^{2}}\right)^{r_{m}},$$
for some $K' > 0$. In~\cite{1987:KarlinOst} it is proved that $(\alpha_{[m]})^{1/m}$ is a decreasing sequence, so we have $\alpha_{[3]}^{1/3}\leq \coincidence^{1/2}$ and hence
$$\left(\frac{\alpha_{[3]}}{\coincidence^{2}}\right)^{r_{m}} \leq \left(\frac{\coincidence^{3/2}}{\coincidence^{2}}\right)^{r_{m}} \leq \coincidence^{-\frac{r_{m}}2}.$$
Therefore
$$\log\left(\frac1N\left(\frac{\alpha_{[3]}}{\coincidence^{2}}\right)^{r_{m}}\right) = -\log N - \frac{r_{m}}2 \log \coincidence \leq -\frac12 \log m + K'' -  \frac{r_{m}}2 \log \coincidence$$
for some constant $K''$. By definition of $r_m$, we have $\log m = -2d' r_m \log\coincidence$ and it follows that
$$\log\left(\frac1N\left(\frac{\alpha_{[3]}}{\coincidence^{2}}\right)^{r_{m}}\right) \leq \frac{r_{m}}2 (2d'-1)\log \coincidence + K''.$$
This quantity tends to $-\infty$ when $n$ tends to $\infty$ since $2d' - 1 > 0$ and $\coincidence<1$.
This proves finally that $\lim_{m\rightarrow\infty}\nu(r_{m})=1$ and establishes Equation~\eqref{eq:lim P(Y)=1}. That is, the desired statement is proved for tuples of infinite reduced words. As $\ell_{n}\leq n$, considering right-infinite words and truncating then at their prefix of length $n$ yields the same result. By construction, the probability distribution induced on this truncated words is exactly $\RR_{n}$, concluding the proof.

The generalisation to cyclically reduced words follows from Lemma~\ref{lemma: conditional probability}.
\end{proof}

We now use Proposition~\ref{pro: collisions} to prove Proposition~\ref{prop: no small cancellation}.

\medskip\noindent\textsc{Proof of Proposition~\ref{prop: no small cancellation}}\enspace
Let $0 < \lambda < \frac12$. Proposition~\ref{pro: collisions}, applied to $\ell_n = \lambda n$ shows that, if $\frac12 < d < 1$, then a random $\coincidence^{-d\lambda n}$-tuple $\vec h$ of cyclically reduced words of length $n$, generically has two components $h_i$ and $h_j$ with the same prefix of length $\lambda n$, which is sufficient to show that $\vec h$ does not satisfy Property $C'(\lambda)$.
\cqfd

We now translate the result of Proposition~\ref{pro: collisions} into a result on the group presented by a random $\coincidence^{-dn}$-tuple, when $d > \frac12$.
We will use repeatedly Chernoff bounds \cite[Th. 4.2 p.70]{1995:MotwaniRaghavan}, which state that, in a binomial distribution with parameters $(k,p)$ --- that is: $X_k$ is the sum of $k$ independent draws of 0 or 1 and $p$ is the probability of drawing 1 ---, 
$$\proba{X_k \leq \frac{kp}{2}} \leq \ \exp\left(-\frac{kp}8\right).$$
In other words, 
\begin{equation}\label{eq:uspensky}
\proba{X_k \ge \frac{kp}2}\geq 1 - \exp\left(-\frac{kp}8\right).
\end{equation}

If $\vec h$ is a vector of cyclically reduced words, $G$ is the group presented by $G = \langle A \mid \vec h\rangle$ and $u,v$ are reduced words, we write that $u =_G v$ if $u$ and $v$ have the same projection in $G$ (that is: if $uv\inv$ lies in the normal closure of $\vec h$).

\begin{proposition}\label{prop: getting degenerate subgroups}
Let $\calA$ be an ergodic Markovian automaton with coincidence probability $\coincidence$ and let $a,b\in \tilde A$ be labels of transitions in $\calA$. Let $d > \frac12$ and $N \ge \coincidence^{-dn}$, and let $G$ be a group presented by an $N$-tuple of cyclically reduced words chosen at random according to $\calA$. Then $a =_G b$ exponentially generically.
\end{proposition}

\begin{proof}
Let $t>0$ be such that all the coefficients of $\MM^t$ are positive (such an integer exists since $\MM$ is ergodic) and let $\tau > 0$ be the minimum coefficient of $\MM^t$. 

We proceed in two steps. First we consider transitions starting in the same state of the Markovian automaton and second we generalize the study to transitions beginning in different states  of the automaton. 

\noindent\textit{First step of the proof.}\enspace
We show that if $x = x_1\cdots x_s$ and $y = y_1\cdots y_s$ are reduced words of equal length $s\ge 1$ which label paths in $\calA$ out of the same state $q$, then exponentially generically, we have $x_k =_G y_k$ for each $1\le k\le s$.

Recall that, in our model of Markovian automata, drawing a word of length $n$ amounts to drawing a state $r\in Q$ according to $\gamma_0$, and then drawing a word of length $n$ according to $\gamma(r,-)$. Thus, when drawing a tuple $\vec h = (h_i)_i$, we also draw a tuple $\vec q = (q_i)_i$ of states such that, in particular, $\gamma_0(q_i) > 0$ and $\gamma(q_i,h_i) > 0$.

Let $r$ be a state such that $\gamma_0(r) > 0$. Let $T_0 =\{h_i \in \vec h \textrm{ such that } q_i=r\}$ and $N_0=|T_0|$.
Observe that drawing randomly and independently $N$ words of length $n$ in our model and then keeping only those starting in state $r$ to obtain $T_0$ is the same as first choosing $N_0$ according to a binomial law of parameters $(\gamma_0(r), N)$ and then drawing randomly and independently $N_0$ words beginning in state $r$. 
Moreover Chernoff bounds (Equation~(\ref{eq:uspensky}) above, applied with $p = \gamma_0(r)$ and $k = N$) show that $\mathbb{P}\left(N_0 \geq \frac{\gamma_0(r) N}2 \right) \geq \gothicp_0$ with $\gothicp_0=1 - \exp\left(-\frac{\gamma_0(r) N}8\right)$.

For each $s \ge 1$, we say that a pair of indices $(i,j)$ is an \emph{$s$-collision in $T_0$} if $h_i$ and $h_j$ belong to $T_0$ and have the same prefix of length $n-t-s$. Let $e$ be such that $0 < e < d-\frac12$ and let $N' = \coincidence^{-(d-e)n}$. Then a random $N_0$-tuple of cyclically reduced words starting in $r$ is obtained by drawing  $\frac{N_0}{N'}$ times a random $N'$-tuple starting in state $r$. 
Moreover choosing a random word in a Markovian automaton given that the associated path begins in state $r$ is the same as taking for initial probability vector $\gamma_0$ the probability vector such that $\gamma_0(r)=1$. 
Since the conclusion of Proposition~\ref{pro: collisions} does not depend on the initial probability vector and $d-e > \frac12$,  Proposition~\ref{pro: collisions}  applied to $\ell_n = n-t-s$ shows that a random $N'$-tuple of cyclically reduced words that starts in $r$ generically exhibits at least one $s$-collision in $T_0$.

We assume that $n$ is large enough so that the probability of an $s$-collision in $T_0$ of a random $N'$-tuple is at least $\frac12$. Then Chernoff bounds (Equation~(\ref{eq:uspensky}), applied with $p = \frac12$ and $k = N_0$) show that the set $T_1$ of $s$-collisions in $T_0$ of a random $\coincidence^{-dn}$-tuple of cyclically reduced words of length $n$ satisfies $|T_1| \ge \frac14 N_0$ with probability greater than or equal to $\gothicp_1 = 1 - \exp(-\frac{N_0}{16})$. 

For each $s$-collision $(i,j) \in T_1$, we let $u(i,j)$ be the common length $n-t-s$ prefix of $h_i$ and $h_j$.
Then by a finiteness argument, there exists a state $q_1 \in Q$ and a set $T_2\subset T_1$ such that, for every $(i,j)\in T_2$, $u(i,j)$ labels a path from $r$ to $q_1$ in $\calA$, and $|T_2| \ge \frac{|T_1|}{|Q|}$. Hence $|T_2| \ge \frac{N_0}{4|Q|}$ with probability greater than or equal to $\gothicp_1$.

Now let $v$ be a reduced word of length $t$, labeling a path in $\calA$ from $q_1$ to $q$: such a word exists since all the coefficients of $\MM^t$ are positive, and we have $\gamma(q_1,v) \ge \tau$. For each $(i,j) \in T_2$, the probability that $h_i$ starts with $u(i,j)v$ is $\gamma(q_1,v) \ge \tau$, and the probability that $uv$ is a prefix of both $h_i$ and $h_j$ is at least $\tau^2$. We can apply Chernoff bounds (\ref{eq:uspensky}) again, with $p = \tau^2$ and $k = |T_2|$: then the subset $T_3 \subseteq T_2$ of pairs $(i,j)$ such that $u(i,j)v$ is a prefix of both $h_i$ and $h_j$, has cardinality $|T_3| \ge \frac12 |T_2| \tau^2$ with probability at least $\gothicp_2 = 1 - \exp(-\frac{\tau^2|T_2|}8)$.

Finally, we note that $|u(i,j)v| = n-s$, so for each $(i,j) \in T_3$, we have $h_i = u(i,j)vx$ with probability $\gamma(q,x)$. Therefore the probability that $(h_i,h_j) = (u(i,j)vx, u(i,j)vy)$ is $\gamma(q,x)\gamma(q,y)$, which is positive by hypothesis. Applying Chernoff bounds one more time (with $k = |T_3|$ and $p = \gamma(q,x)\gamma(q,y)$) shows that $\vec h$ contains a pair of words of the form $(wx,wy)$ with probability at least
$\gothicp_3$ with $\gothicp_3 =\left(1 -  \exp\left(-\frac{|T_3|\gamma(q,x)\gamma(q,y)}8\right)\right)$.

In conclusion, exponentially generically $N_0\geq \gamma_0(r) \frac{N}2$ which implies that $\gothicp_1$ is exponentially close to $1$. Hence $T_2\geq \frac{\gamma_0(r)N}{8|Q|}$ exponentially generically, which implies that $\gothicp_2$ is exponentially close to $1$. So  $|T_3| \ge \frac{\gamma_0(r)N  \tau^2}{16|Q|}$  exponentially generically,  which implies that $\gothicp_3$ is exponentially close to $1$. In particular, exponentially generically, $\vec h$ has a pair of the form $(wx,wy)$, and hence we have $x =_G y$.

Applying this to the words $x_1$ and $y_1$, we find that $x_1 =_G y_1$. Next, considering the words $x_1x_2$ and $y_1y_2$, we find that $x_1x_2 =_G y_1y_2$, and hence $x_2 =_G y_2$. Iterating this reasoning, we finally show that $
x_k =_G y_k$ for each $1 \le k \le s$.

\smallskip\noindent\textit{Second step of the proof}\enspace
We now consider two transitions in $\calA$, one labeled $a$ from state $q$ to state $q'$ and another labeled $b$ from state $r$ to state $r'$ ($a,b\in \tilde A$).

Let $q_0\in Q$ be a state in $\calA$ such that $\gamma_0(q_0) > 0$. Since $\calA$ is irreducible, there exists a word $w_1$ which labels a loop at $q_0$ and visits every transition of $\calA$. Moreover, since $\calA$ is ergodic, there exists a word $w_2$ labeling another loop at $q_0$, such that $|w_1|$ and $|w_2|$ are relatively prime.

Since reading $w_1$ from $q_0$ visit all the transitions, let $u_1$ (resp. $v_1$) be a prefix of $w_1$ such that the last transition read after reading $u_1$ (resp. $v_1$) is the $a$-transition out of state $q$ (resp. the $b$-transition out of state $r$).	 Then the Chinese remainder theorem shows that there exist words $x\in \{w_1,w_2\}^*u_1$ and $y\in \{w_1,w_2\}^*v_1$ of equal length.

Since $a$ and $b$ are the last letters of $x$ and $y$, respectively, the first step of the proof shows that $a =_G b$, which concludes the proof of the proposition.
\eop

\noindent We can now complete the {\bf proof of Proposition~\ref{prop: degenerate subgroups}.}
By Proposition~\ref{prop: getting degenerate subgroups}, exponentially generically, all the letters in $E$ are equal in $G$. If $a, a\inv \in E$ for some letter $a$, then all these letters are equal to their own inverse in $G$, so the subgroup $H$ of $G$ generated by $E$ is a quotient of $\ZZ/2\ZZ$. Since all the relators in the presentation have length $n$, it follows that $H$ is isomorphic to $\ZZ/2\ZZ$ if $n$ is even, and is trivial if $n$ is odd. The result follows once we observe that the letters in $D$ do not occur in any relator.
\cqfd

\subsubsection*{Acknowledgments}
The authors are thankful to the anonymous referee for her/ his remarkably thorough reading of the first version of this paper and for his/her insightful and constructive suggestions. These helped simplify the presentation of Sections~\ref{sec: repeated factors} and~\ref{sec: repeated cyclic factors}, sharpen some results in Section~\ref{sec: applications to uniform} and fix a technical mistake in the proof of Proposition~\ref{prop: degenerate subgroups}.

\bibliographystyle{abbrv}

\begin{thebibliography}{}

\end{thebibliography}


\begin{thebibliography}{10}

\bibitem{1996:ArzhantsevaOlshanskii}
G.~N. Arzhantseva and A.~Y. Ol'shanski{\u\i}.
\newblock Generality of the class of groups in which subgroups with a lesser
  number of generators are free.
\newblock {\em Mat. Zametki}, 59(4):489--496, 638, 1996.

\bibitem{2012:BassinoNicaudWeil}
F.~Bassino, C.~Nicaud, and P.~Weil.
\newblock Generic properties of random subgroups of a free group for general distributions.
\newblock In {\em 23rd {I}ntern. {M}eeting on {P}robabilistic, {C}ombinatorial, and {A}symptotic {M}ethods for the {A}nalysis of {A}lgorithms ({A}of{A}'12)}, Discrete Math. Theor. Comput. Sci. Proc., AQ, pages 155--166.
Assoc. Discrete Math. Theor. Comput. Sci., Nancy, 2012.

\bibitem{2013:BassinoMartinoNicaud}
F.~Bassino, A.~Martino, C.~Nicaud, E.~Ventura, and P.~Weil.
\newblock Statistical properties of subgroups of free groups.
\newblock {\em Random Structures Algorithms}, 42(3):349--373, 2013.

\bibitem{2008:BassinoNicaudWeil}
F.~Bassino, C.~Nicaud, and P.~Weil.
\newblock Random generation of finitely generated subgroups of a free group.
\newblock {\em Internat. J. Algebra Comput.}, 18(2):375--405, 2008.

\bibitem{2015:BassinoNicaudWeil}
F.~Bassino, C.~Nicaud, and P.~Weil.
\newblock On the genericity of {W}hitehead minimality.
\newblock {\em J. Group Theory}, to appear, 2015.

\bibitem{2001:BridsonWise}
M.~R. Bridson and D.~T. Wise.
\newblock Malnormality is undecidable in hyperbolic groups.
\newblock {\em Israel J. Math.}, 124:313--316, 2001.

\bibitem{1995:Champetier}
C.~Champetier.
\newblock Propri{\'e}t{\'e}s statistiques des groupes de pr{\'e}sentation
  finie.
\newblock {\em Journal of Advances in Mathematics}, 116(2):197--262, 1995.

\bibitem{1959:Gantmacher}
F.~R. Gantmacher.
\newblock {\em The theory of matrices}.
\newblock Chelsea, 1959.

\bibitem{1998:GitikMitraRips}
R.~Gitik, M.~Mitra, E.~Rips, and M.~Sageev.
\newblock Widths of subgroups.
\newblock {\em Trans. Amer. Math. Soc.}, 350(1):321--329, 1998.

\bibitem{1987:Gromov}
M.~Gromov.
\newblock Hyperbolic groups.
\newblock In {\em Essays in group theory}, volume~8 of {\em Math. Sci. Res.
  Inst. Publ.}, pages 75--263. Springer, New York, 1987.

\bibitem{1993:Gromov}
M.~Gromov.
\newblock Asymptotic invariants of infinite groups.
\newblock In {\em Geometric group theory, {V}ol.\ 2 ({S}ussex, 1991)}, volume
  182 of {\em London Math. Soc. Lecture Note Ser.}, pages 1--295. Cambridge
  Univ. Press, Cambridge, 1993.

\bibitem{2002:Jitsukawa}
T.~Jitsukawa.
\newblock Malnormal subgroups of free groups.
\newblock In {\em Computational and statistical group theory ({L}as {V}egas,
  {NV}/{H}oboken, {NJ}, 2001)}, volume 298 of {\em Contemp. Math.}, pages
  83--95. Amer. Math. Soc., Providence, RI, 2002.

\bibitem{2015:Kapovich}
I.~Kapovich.
\newblock Musings on generic-case complexity.
\newblock \url{arXiv:1505.03218}, 2015.

\bibitem{2003:KapovichMiasnikovSchupp}
I.~Kapovich, A.~Miasnikov, P.~Schupp, and V.~Shpilrain.
\newblock Generic-case complexity, decision problems in group theory, and
  random walks.
\newblock {\em J. Algebra}, 264(2):665--694, 2003.

\bibitem{2002:KapovichMyasnikov}
I.~Kapovich and A.~Myasnikov.
\newblock Stallings foldings and subgroups of free groups.
\newblock {\em J. Algebra}, 248(2):608--668, 2002.

\bibitem{1987:KarlinOst}
S.~Karlin and F.~Ost.
\newblock Counts of long aligned word matches among random letter sequences.
\newblock {\em Adv. in Appl. Probab.}, 19(2):293--351, 1987.

\bibitem{1998:KharlampovichMyasnikov}
O.~Kharlampovich and A.~Myasnikov.
\newblock Hyperbolic groups and free constructions.
\newblock {\em Trans. Amer. Math. Soc.}, 350(2):571--613, 1998.

\bibitem{2009:LevinPeresWilmer}
D.~A. Levin, Y.~Peres, and E.~L. Wilmer.
\newblock {\em Markov chains and mixing times}.
\newblock American Mathematical Society, Providence, RI, 2009.

\bibitem{1977:LyndonSchupp}
R.~C. Lyndon and P.~E. Schupp.
\newblock {\em Combinatorial group theory}.
\newblock Springer-Verlag, Berlin, 1977.
\newblock Ergebnisse der Mathematik und ihrer Grenzgebiete, Band 89.

\bibitem{1995:MotwaniRaghavan}
R.~Motwani and P.~Raghavan.
\newblock {\em Randomized Algorithms}.
\newblock Cambridge University Press, 1995.
 

\bibitem{2007:MiasnikovVenturaWeil}
A.~Miasnikov, E.~Ventura, and P.~Weil.
\newblock Algebraic extensions in free groups.
\newblock In {\em Geometric group theory}, Trends Math., pages 225--253.
  Birkh\"auser, Basel, 2007.

\bibitem{1918:Nielsen}
J.~Nielsen.
\newblock Die {I}somorphismen der allgemeinen, unendlichen {G}ruppe mit zwei
  {E}rzeugenden.
\newblock {\em Mathematische Annalen}, 78, 1918.

\bibitem{2004:Ollivier}
Y.~Ollivier.
\newblock Sharp phase transition theorems for hyperbolicity of random groups.
\newblock {\em Geom. Funct. Anal.}, 14(3):595--679, 2004.

\bibitem{2005:Ollivier}
Y.~Ollivier.
\newblock {\em A {J}anuary 2005 invitation to random groups}, volume~10 of {\em
  Ensaios Matem{\'a}ticos [Mathematical Surveys]}.
\newblock Sociedade Brasileira de Matem{\'a}tica, 2005.

\bibitem{1992:Olshanskii}
A.~Y. Ol'shanski{\u\i}.
\newblock Almost every group is hyperbolic.
\newblock {\em Internat. J. Algebra Comput.}, 2(1):1--17, 1992.

\bibitem{1963:Rabin}
M.~O. Rabin.
\newblock Probabilistic automata.
\newblock {\em Information and Computation}, 6(3):230--245, 1963.

\bibitem{2007:RoigVenturaWeil}
A.~Roig, E.~Ventura, and P.~Weil.
\newblock On the complexity of the {W}hitehead minimization problem.
\newblock {\em Internat. J. Algebra Comput.}, 17(8):1611--1634, 2007.

\bibitem{1995:SegalaLynch}
R.~Segala and N.~Lynch.
\newblock Probabilistic simulations for probabilistic processes.
\newblock {\em Nordic Journal of Computing}, 2(2):250--273, 1995.

\bibitem{1980:Serre}
J.-P. Serre.
\newblock {\em Trees}.
\newblock Springer-Verlag, Berlin, 1980.
\newblock Translated from the French by John Stillwell.

\bibitem{2008:SilvaWeil}
P.~V. Silva and P.~Weil.
\newblock On an algorithm to decide whether a free group is a free factor of
  another.
\newblock {\em Theor. Inform. Appl.}, 42(2):395--414, 2008.

\bibitem{1983:Stallings}
J.~R. Stallings.
\newblock Topology of finite graphs.
\newblock {\em Invent. Math.}, 71(3):551--565, 1983.

\bibitem{1991:Szpankowski}
W.~Szpankowski.
\newblock On the height of digital trees and related problems.
\newblock {\em Algorithmica}, 6(2):256--277, 1991.

\bibitem{2006:Touikan}
N.~W.~M. Touikan.
\newblock A fast algorithm for {S}tallings' folding process.
\newblock {\em Internat. J. Algebra Comput.}, 16(6):1031--1045, 2006.

\bibitem{2009:ValleeClementFillFlajolet}
  B.~Vall{\'e}e, J.~Cl{\'e}ment, J.~A.~Fill and Ph.~Flajolet.
  \newblock The number of symbol comparisons in QuickSort and QuickSelect.
  \newblock  {\em Automata, Languages and Programming}, Springer, pages 750--763. Berlin Heidelberg, 2009. 

\bibitem{2000:Weil}
P.~Weil.
\newblock Computing closures of finitely generated subgroups of the free group.
\newblock In {\em Algorithmic problems in groups and semigroups ({L}incoln,
  {NE}, 1998)}, Trends Math., pages 289--307. Birkh\"auser Boston, Boston, MA,
  2000.


\end{thebibliography}

\end{document}